\documentclass[11pt]{article} 
\usepackage{caption}
\usepackage[utf8]{inputenc} 
\usepackage{amsmath}
\usepackage{amssymb}
\usepackage{floatpag}
\usepackage[toc,page]{appendix}

\usepackage{ulem}

\usepackage{amsthm}
\usepackage{ mathrsfs }
\usepackage{placeins}
\usepackage{titlesec}
\titleformat{\paragraph}
{\normalsize\bfseries}{\theparagraph}{1em}{}
\titlespacing*{\paragraph}
{0pt}{3.25ex plus 1ex minus .2ex}{1.5ex plus .2ex}

\titleclass{\subsubsubsection}{straight}[\subsection]

\newcounter{subsubsubsection}[subsubsection]
\renewcommand\thesubsubsubsection{\thesubsubsection.\arabic{subsubsubsection}}
\renewcommand\theparagraph{\thesubsubsubsection.\arabic{paragraph}} 

\titleformat{\subsubsubsection}
  {\sffamily\small}{\thesubsubsubsection}{1em}{}
\titlespacing*{\subsubsubsection}
{0pt}{3.25ex plus 1ex minus .2ex}{1.5ex plus .2ex}

\newcommand{\RN}[1]{%
  \textup{\uppercase\expandafter{\romannumeral#1}}%
}
\usepackage{titlesec}
\titleformat{\section}{\normalfont\Large\bfseries}{\S\thesection}{1em}{}[]

\newtheorem{Theorem}{Theorem}[section]
\newtheorem{Corollary}[Theorem]{Corollary}
\newtheorem{Lemma}[Theorem]{Lemma}
\newtheorem{Definition}[Theorem]{Definition}
\newtheorem{Remark}[Theorem]{Remark}

\usepackage{float}
\restylefloat{table}
\usepackage{dblfloatfix}
\usepackage[hidelinks]{hyperref}
\usepackage{color}
\usepackage[table,xcdraw]{xcolor}

\usepackage{placeins}
  \usepackage{graphicx}
    \usepackage{caption}    
    \usepackage{array,booktabs} \usepackage{amsthm,amsmath,amsfonts,amssymb,latexsym,amscd,mathrsfs}
    \usepackage{textcomp,float}
    \usepackage[margin=25mm]{geometry}

    \usepackage{tikz}
    \usepackage{tikz-3dplot}
    \usetikzlibrary{shapes,calc,positioning}
\usepackage{afterpage}
\usepackage{placeins}
\usepackage{hyperref}
\def\bF{\mathbb{F}}
\def\Fq{\mathbb{F}_q}
\def\GL{\mathrm{GL}}
\def\PG{\mathrm{PG}}
\def\PGL{\mathrm{PGL}}
\usepackage{graphicx} 
\definecolor{amber(sae/ece)}{rgb}{1.0, 0.49, 0.0}
	\definecolor{darkcyan}{rgb}{0.0, 0.55, 0.55}
	\definecolor{darkseagreen}{rgb}{0.56, 0.74, 0.56}
	\definecolor{salmon}{rgb}{1.0, 0.55, 0.41}
\usepackage{booktabs} 
\usepackage{array} 
\usepackage{paralist} 
\usepackage{verbatim} 
\usepackage{subfig} 
\definecolor{ForestGreen}{RGB}{34,139,34}
\usepackage{fancyhdr} 
\definecolor{carminepink}{rgb}{0.92, 0.3, 0.26}
	\definecolor{electricyellow}{rgb}{1.0, 1.0, 0.0}
	\definecolor{chromeyellow}{rgb}{1.0, 0.65, 0.0}
	\definecolor{babyblue}{rgb}{0.54, 0.81, 0.94}

\usepackage{sectsty}
\allsectionsfont{\sffamily\mdseries\upshape} 

 \def\bF{\mathbb{F}}
\def\Fq{\mathbb{F}_q}
\def\GL{\mathrm{GL}}

\def\F2h{\mathbb{F}_{2^h}}

\def\PG{\mathrm{PG}}
\def\PGL{\mathrm{PGL}}

\def\cP{\mathcal{P}}

\def\cV{\mathcal{V}}
\def\cZ{\mathcal{Z}}
\def\cC{\mathcal{C}}
\def\cH{\mathcal{H}}

\def\cH{\mathcal{H}}

\usepackage[nottoc,notlof,notlot]{tocbibind} 
\usepackage[titles,subfigure]{tocloft}

\title{Webs and squabs of conics over finite fields}

\usepackage{placeins}

\author{Nour Alnajjarine\footnote{University  of Rijeka; Lebanese International University (\texttt{alnajjarine@gmail.com})} \;  and \;
Michel Lavrauw\footnote{University of Primorska; Vrije Universiteit Brussel (\texttt{michel.lavrauw@famnit.upr.si})}}

\begin{document} 
\maketitle
\begin{abstract}
This paper is a contribution towards a solution for the longstanding open problem of classifying linear systems of conics over finite fields initiated by L. E. Dickson in 1908, through his study of the projective equivalence classes of pencils of conics in $\PG(2,q)$, for $q$ odd. In this paper a set of complete invariants is determined for the projective equivalence classes of webs and of squabs of conics in $\PG(2,q)$, both for $q$ odd and even.
Our approach is mainly geometric, and involves a comprehensive study of the geometric and combinatorial properties of the Veronese surface in $\PG(5,q)$. The main contribution is the determination of the distribution of the different types of hyperplanes incident with the $K$-orbit representatives of points and lines of $\PG(5,q)$, where $K\cong \PGL(3,q)$, is the subgroup of $\PGL(6,q)$ stabilizing the Veronese surface.  
\end{abstract}
\textbf{Keywords:}\hspace{0.2cm}Veronese Surface,\hspace{0.2cm} Linear Systems,\hspace{0.2cm} Webs of Conics,\hspace{0.2cm} Squabs of Conics,\hspace{0.2cm} Finite Fields, \hspace{0.2cm} Galois Geometry.

\section{Introduction}

Let ${\mathcal{F}}_d(n,\mathbb{F})$ be the space of forms of degree $d$ defined on the $n$-dimensional projective space $\PG(n, \mathbb{F})$. {\it Linear systems} of hypersurfaces of degree $d$ are subspaces of the projective geometry $\PG({\mathcal{F}}_d(n,\mathbb{F}))$ associated with ${\mathcal{F}}_d(n,\mathbb{F})$. In particular, $1$, $2$, 
$3$ and $4$-dimensional subspaces of 
$\PG({\mathcal{F}}_2(2,\mathbb{F}))$ are called pencils, nets,  webs and squabs of conics. The study of linear systems of conics enjoys a long and rich history, dating back to the late-$19$th century. Scholars such as C. Jordan were among the first to investigate these systems, exploring their properties over both real and complex numbers \cite{jordan1,jordan2}, while C. Segre made significant contributions in the study of pencils of quadrics in higher dimensions \cite{segre}, particularly focusing on algebraically closed fields with characteristics different from two. Their pioneering work laid the foundation for subsequent research at the start of the 20th century on linear systems of conics (including over finite fields) by Dickson, Campbell, Wilson, Wall, and others \cite{ campbell,campbell2,wilson,wall,lines,nets,solidsqeven,planesqeven}. We direct the reader to \cite{lines} for a detailed elucidation on the insufficiency of the elementary divisor method employed by C. Segre and fellow scholars in addressing the finite field case, contrasting its efficacy in the case of algebraically closed fields.
We also refer to \cite{AbEmIa}, which considers nets over algebraically closed fields of characteristic zero and their relation to Artinian algebras of length 7, and which includes an interesting historical account on nets of conics
in the appendices A and B.

\bigskip

The problem of classifying linear systems of conics was first studied by Jordan and Wall who classified pencils and nets in \cite{jordan1,jordan2,wall} over $\mathbb{R}$ and $\mathbb{C}$.  For finite fields, the narrative is more complicated. We direct readers to references \cite{lines} and \cite{solidsqeven} for historical remarks and a discussion of the findings concerning pencils of conics, initially investigated by Dickson over finite fields of odd characteristic and subsequently by Campbell for even characteristic. Dickson's paper, published in 1908, contains a thorough study and a complete classification of pencils of conics over finite fields of odd order, using a purely algebraic approach. Campbell's paper, published in 1927, does not give a complete classification, see \cite[Section 1.2]{lines} for further details. Nets over finite fields of odd order were studied by Wilson in 1914 \cite{wilson}, and for even order, these nets were studied by Campbell in 1928 \cite{campbell2}. Although both studies have their merits, neither provides a complete classification of nets of conics. 
For the odd characteristic case, we refer to \cite[Section 9]{nets} for a classification of nets of rank one and for an explanation of some of the shortcomings of Wilson's treatment. For the even characteristic case we refer to \cite{planesqeven} were we  pointed out 
 some errors in Campbell’s work, and   classified nets of conics with non-empty base.

\bigskip

The history on the problem of classifying linear systems of conics over finite fields, has made it evident that a purely computational approach will unlikely lead to much further progress. In particular, in order to complete the classification of nets of conics over finite fields, a more detailed study of the different linear systems and their interrelations is needed. This need for a new approach has triggered the interest in extending our current understanding of linear systems of conics, beyond merely classifying them. One of these approaches involves the distinction between projectively inequivalent systems using a set of geometric and combinatorial invariants.
Powerful combinatorial invariants (called {\it orbit distributions}, see Section \ref{pre} below) introduced in \cite{lines} and \cite{solidsqeven}, were used to determine the number of the different types of conics contained in pencils of conics in $\PG(2,q)$. In \cite{LaPoSh2021} it was shown that, in the case of nets of rank one over finite field of odd order, these combinatorial invariants give in fact a complete invariant for the equivalence classes of these nets.

\bigskip

In addition to the above mentioned new ideas, and in contrast to the purely algebraic method employed by Jordan, Dickson, Campbell, and Wilson, our approach is geometric and combinatorial, and is based on the correspondence between linear systems of conics in $\PG(2,q)$ and subspaces of the ambient space $\PG(5,q)$ of the quadric Veronesean. This connection is well-known and classical, and will be outlined in Section \ref{pre}. The study of linear systems of conics in $\PG(2,q)$ up to projective equivalence is equivalent to the study of $K$-orbits of subspaces of $\PG(5, q)$, where $K$ is the subgroup of $\PGL(6,q)$ obtained by lifting the natural action of $\PGL(3,q)$ on $\PG(2,q)$ through the Veronese map from $\PG(2,q)$ to $\PG(5,q)$. The proofs of the results in this paper are a combination of geometric, algebraic, and combinatorial techniques, and some ideas came from initial computational data obtained through the GAP \cite{GAP} package FinInG \cite{fining}.

\bigskip

In this paper, the combinatorial and geometric properties of the quadric Veronesean in $\PG(5,q)$ are used to determine the number of different types of conics contained in squabs and webs of conics in $\PG(2,q)$. These numbers serve as combinatorial invariants for webs and squabs of conics, and they are determined by studying the different types of hyperplanes ($K$-orbits) incident with their associated points and lines of $\PG(5,q)$. Note that the $K$-orbits on points of $\PG(5,q)$ are well-known and quite simple, while $K$-orbits of lines of $\PG(5,q)$ were classified in \cite{lines}, based on the canonical forms of tensors of format $(2,3,3)$ from \cite{canonical}.
Further, we prove in Theorem  \ref{webs2} that the intersection of a line of $\PG(5,q)$ with the secant variety of the Veronese surface completely determines the singularity of its associated web of conics in $\PG(2,q)$. Consequently, we obtain a number of characterisations of squabs and webs of conics over finite fields  highlighted in Sections \ref{squabs} and \ref{webs}. In particular, we show in Corollary \ref{cor:squabs} that any two squabs of conics in $\PG(2,q)$ having the same number of double lines and the same number of non-singular conics are projectively equivalent.

\bigskip

\par The paper is structured as follows. The proofs of our main results are given in Sections  \ref{squabs} and \ref{webs}. In particular, we  determine in Section \ref{squabs} the number of different types
of conics contained in squabs of conics in $\PG(2,q)$. We summarize these invariants in Tables \ref{table_hypOD_points} and \ref{table_hypOD_points_even}. In Section \ref{webs}, we determine the number of different types
of conics contained in webs of conics in $\PG(2,q)$. We summarize in Tables \ref{tablewebsodd} and \ref{tablewebseven} the $15$ webs of conics in $\PG(2,q)$ up to projective equivalence, along with their associated $K$-orbits of lines in $\PG(5,q)$ and their hyperplane-orbit distributions. Representatives of the $K$-orbits of lines and their  point-orbit distributions are listed in Tables \ref{tableoflinesodd} 
 and \ref{tableoflineseven}. We list in Table \ref{solidsodd}  the $15$ $K$-orbits of solids in $\PG(5,q)$, $q$ odd. In Tables \ref{OD2(H)odd} and \ref{OD2(H)even}, we collect line-orbit distributions of hyperplanes of $\PG(5,q)$. We refer to Section \ref{pre} for the definition of point, line and hyperplane-orbit distributions.

\section{Preliminaries}\label{pre}

In this section, we collected the definitions and theory needed in our proofs. Some of the results are classical and can be found in e.g. \cite{galois geometry,hirsch}, others are more recent results from \cite{T233paper,solidsqeven,planesqeven,roots, lines,nets,LaPoSh2021}.

\subsection{Linear systems of conics}

The finite field of order $q$ is denoted by $\bF_q$, where $q=p^h$, $p$ prime. 
The set of squares in $\bF_q$ is denoted by $\square_q$. The $n$-dimensional projective space over the finite field with $q$ element is denoted by $\PG(n,q)$. The projectivization of a vector space $V$ is denoted by $\PG(V)$. A {\it form on $\PG(n,q)$} is a homogeneous polynomial in $\bF_q[X_0,\ldots,X_n]$, and by $\cZ(g_1,\ldots,g_k)$ we denote the algebraic variety in $\PG(n,q)$ defined as the common zero locus in $\PG(n,q)$ of the forms $g_1,\ldots g_k$ on $\PG(n,q)$. Two algebraic varieties of $\PG(n,q)$ are {\it projectively equivalent} if there exists a projectivity (an element of the projective linear group $\PGL(n+1,q)$) mapping one to the other.
An algebraic variety $\cZ(g)$ defined by a single form $g$ on $\PG(n,q)$ is called a {\it hypersurface} in $\PG(n,q)$, and a {\it linear system} of hypersurfaces in $\PG(n,q)$ is the set of hypersurfaces 
$${\mathcal{LS}}(g_1,\ldots,g_k)=\left \{\cZ\left (\sum_{i=1}^k a_i g_i\right )~:~a_1,\ldots,a_k\in \bF_q\right \}$$ 
defined by forms $g_1,\ldots,g_k$ on $\PG(n,q)$ of the same degree.
The set of forms on $\PG(n,q)$ of degree $d$ form a vector subspace ${\mathcal{F}}_d(n,q)$ of $\bF_q[X_0,\ldots,X_n]$. 
A linear system ${\mathcal{LS}}(g_1,\ldots,g_k)$ is called {\it trivial} if the forms $g_1,\ldots,g_k\in {\mathcal{F}}_d(n,q)$ are all zero or if they span all of ${\mathcal{F}}_d(n,q)$.
We define the 
{\it dimension} of a linear system ${\mathcal{LS}}(g_1,\ldots,g_k)$ as the (projective) dimension of the subspace $\langle g_1,\ldots,g_k\rangle$ in $\PG({\mathcal{F}}_d(n,q))$. Linear systems of dimension  1, 2, and 3 are called {\it pencils}, {\it nets}, and {\it webs}. A linear system of hypersurfaces of co-dimension one, is called a {\it squab} of hypersurfaces.\\

A {\it conic} is defined as the zero locus $\cZ(f)$ of a quadratic form $f=a_{00}X^2+a_{01}XY+a_{02}XZ+a_{11}Y^2+a_{12}YZ+a_{22}Z^2$ on $\PG(2,q)$. The conic $\cZ(f)$ and the quadratic form $f$ have {\it discriminant} $$\Delta_{f}=4a_{00}a_{11}a_{22}+a_{01}a_{02}a_{12}-a_{00}a_{12}^2-a_{11}a_{02}^2-a_{22}a_{01}^2.$$
A conic is {\it non-singular} (or {\it non-degenerate}) if its discriminant is nonzero. There are four types of conics in $\PG(2,q)$: double lines, pairs of distinct lines, pairs of conjugate (imaginary) lines over $\bF_{q^2}$, and non-singular conics. Each two conics of the same type are projectively equivalent. Recall that a polynomial $f(X)=\alpha X^2 + \beta X + \gamma \in \mathbb{F}_{q}[X]$, $q$ even, with $\alpha \neq 0$ has exactly one root in $\mathbb{F}_{q}$ if and only if $\beta = 0$, two distinct roots in $\mathbb{F}_{q}$ if and only if $\beta \neq 0$ and $\operatorname{Tr}(\frac{\alpha \gamma}{\beta^2})=0$, and no roots in $\mathbb{F}_{q}$ otherwise, where $\operatorname{Tr}$ denotes the trace map from $\mathbb{F}_{q}$ to $\mathbb{F}_p$.
Since the space of quadratic forms in three variables has dimension six, a non-trivial linear system of conics, is either a conic, a pencil, a net, a web, or a squab of conics. 

\subsection{The Veronese surface in $\PG(5,q)$}

The {\it Veronese surface} $\cV(\bF_q)$ is a 2-dimensional algebraic variety in $\PG(5,q)$ which can be defined as the image of the {\it Veronese map}
\[
\nu: \PG(2,q)\rightarrow \PG(5,q) \quad \text{given by} \quad
(u_0,u_1,u_2) \mapsto (u_0^2,u_0u_1,u_0u_2,u_1^2,u_1u_2, u_2^2).
\]
The Veronese surface $\cV(\Fq)$ is also known as the {\it quadric Veronesean}.
We adopt the notation from \cite{nets,LaPoSh2021}, where a point $P=(y_0,y_1,y_2,y_3,y_4,y_5)$ of $\PG(5,q)$ is represented by a symmetric $3 \times 3$ matrix
\[
M_P=\begin{bmatrix} y_0&y_1&y_2\\
y_1&y_3&y_4\\
y_2&y_4&y_5  \end{bmatrix},
\]
and in which a dot represents a zero.
This representation is extended to subspaces of $\PG(5,q)$, see e.g. the tables in the Appendix.
The {\it rank of a point} $P\in\PG(5,q)$ is defined as  the rank of its associated matrix $M_P$. 
The points of rank $1$ are precisely those belonging to $\mathcal{V}(\Fq)$. Throughout the paper, the indeterminates of the coordinate rings of $\PG(2,q)$ and $\PG(5,q)$ are denoted by $(X_0,X_1,X_2)$ and $(Y_0,\ldots,Y_5)$ respectively. 
The Veronese surface $\cV(\Fq)$ contains $q^2 + q + 1$ conics, defined as the image of lines in $\PG(2,q)$ under the Veronese map $\nu$, where any two points $P$, $Q$ of $\mathcal{V}(\Fq)$ lie on a unique such conic given by 
\[
\mathcal{C}(P,Q):= \nu(\langle\nu^{-1}(P),\nu^{-1}(Q)\rangle),
\] 
and any two of these conics have a unique point in common. 
\begin{Remark}
Note that when $q=2$, each three points on $\cV(\bF_q)$ form a conic, so, in order to include the case $q=2$ in our arguments for general prime powers $q$, we need to somewhat alter the definition of a {\it conic of $\cV(\bF_q)$} in the case $q=2$, by defining such a conic as the image of a line of $\PG(2,q)$ under the Veronese map $\nu$. Other triplets of points on $
\cV(\bF_q)$ are not considered as conics in $\cV(\bF_q)$.
\end{Remark}
The conics of $\PG(2, q)$ correspond to the hyperplane sections of $\cV(\Fq)$. Planes of $\PG(5, q)$ which meet $\cV(\Fq)$ in a conic are called {\it conic planes}.
 If $q$ is even, then all tangent lines to a conic $\mathcal{C}$ in $\mathcal{V}(\Fq)$ are concurrent, meeting at the {\it nucleus} of $\cC$. 
The set of all nuclei of conics in $\mathcal{V}(\Fq)$ coincides with the set of points of a plane in $\PG(5,q)$ known as the {\it nucleus plane} $\pi_{\mathcal{N}}$ of $\cV(\bF_q)$. In the above matrix representation, points contained in the nucleus plane correspond to symmetric $3\times 3$ matrices with zeros on the main diagonal, i.e. $\pi_{\mathcal{N}}=\mathcal{Z}(Y_0, Y_3, Y_5)$.
Every rank-$2$ point $R$ of $\PG(5,q)$ defines a unique conic $\mathcal{C}(R)$ in $\cV(\bF_q)$. 
The {\it tangent lines to $\cV(\Fq)$} are defined as the tangent lines to the conics in $\cV(\Fq)$. The tangent lines of $\cV(\Fq)$ at a point $P\in \cV(\Fq)$ are contained in {\it the tangent plane of $\cV(\Fq)$ at $P$}.

\subsection{The group action}
The action on subspaces of $\PG(5,q)$ by the group $K\leqslant \PGL(6,q)$ is defined as the lift of $\PGL(3,q)$ through the Veronese map $\nu$. 
Explicitly, if $\phi_A\in \PGL(3,q)$ is the projectivity defined by the matrix $A \in \GL(3,q)$ then we define the corresponding projectivity $\alpha(\phi_A)\in \PGL(6,q)$ through its action on the points of $\PG(5,q)$ by 
\[
\alpha(\phi_A): P\mapsto Q \quad \text{where} \quad M_Q=AM_PA^T,
\] 
where $M_Q$ and $M_P$ are the matrix representations of $Q$ and $P$. Then $K:=\alpha(\PGL(3,q))$ is isomorphic to $\PGL(3,q)$ and leaves $\cV(\bF_q)$ invariant. 
Note that if $q > 2$ then $K \cong \PGL(3,q)$ is the {\it full} setwise stabiliser of $\cV(\bF_q)$ in $\PGL(6,q)$, while if $q=2$ then the full setwise stabiliser of $\cV(\bF_q)$ is $\text{Sym}_7$.
 The {\it rank distribution} of a subspace $W$ of $\PG(5,q)$ is the $3$-tuple $[r_1,r_2,r_3]$, where $r_i$ represents the count of points of rank $i$ in $W$. The rank distribution of a subspace is $K$-invariant. The next definition gives a stronger $K$-invariants, first introduced in \cite{LaPoSh2021,solidsqeven}. They play a pivotal  role in classifying $K$-orbits of subspaces of $\PG(5,q)$.
 
 \begin{Definition}\label{orbitdist}
\textnormal{
Let $W_1,W_2,\dots,W_m$ be the list of $K$-orbits of $r$-dimensional subspaces in $\PG(n,q)$. 
The \textit{$r$-space orbit distribution} of a subspace $W$ of $\PG(n,q)$ (with respect to this ordering of $K$-orbits) is the list $OD_r(W)=[w_1,w_2,\ldots,w_m]$, where $w_i=w_i(W)$ is the number of elements of the orbit $W_i$ incident with $W$.
}
\end{Definition}
There are four $K$-orbits of points in $\PG(5,q)$: the orbit $\mathcal{P}_1=\cV(\Fq)$ of rank-$1$ points, which has size $q^2+q+1$, the orbit of rank-$3$ points $\mathcal{P}_3$, which has size $q^5-q^2$, and two orbits of rank-$2$ points. For $q$ even, the orbits of rank-$2$ points comprise the $q^2+q+1$ points of the nucleus plane $\pi_{\mathcal{N}}$, denoted by $\mathcal{P}_{2,n}$, and the $(q^2-1)(q^2+q+1)$ points contained in conic planes but not in $\pi_{\mathcal{N}}\cup \cV(\bF_q) $, denoted by $\mathcal{P}_{2,s}$.
For $q$ odd, the orbits of rank-$2$ points comprise the $(q^2+q+1)q(q+1)/2$ exterior points, denoted by $\mathcal{P}_{2,e}$, and $(q^2+q+1)q(q-1)/2$ interior points, denoted by $\mathcal{P}_{2,i}$.
Therefore, the point-orbit distribution of a subspace $W$ of $\PG(5,q)$ is the $4$-tuple $OD_0(W)=[r_1,r_{2,n},r_{2,s},r_3]$ for $q$ even and $OD_0(W)=[r_1,r_{2,e},r_{2,i},r_3]$ for $q$ odd, where $r_i=r_i(W)$, $i\in \{1,3\}$, is the number of rank-$i$ points in $W$, $r_{2,n}=r_{2,n}(W)$ is the number of rank-$2$ points in $W\cap\pi_{\mathcal{N}}$, $r_{2,s}=r_{2,s}(W)$ is the number of rank-$2$ points in $W\setminus \pi_{\mathcal{N}} $, $r_{2,e}=r_{2,e}(W)$ is the number of exterior rank-$2$ points in $W$, and $r_{2,i}=r_{2,i}(W)$ is the number of interior rank-$2$ points in $W$. Similarly, we obtain the {\it line-}, {\it plane-}, {\it solid-}, and {\it hyperplane-orbit distributions} of $W$ for $r=1,2,3,4$ respectively. This data turns out to be very useful in determining the $K$-orbit of a subspace of $\PG(5,q)$. 
The line orbits themselves were determined (for all $q$) in \cite{lines}, based on the canonical forms of tensors of format $(2,3,3)$ from \cite{canonical}.

\begin{Theorem}\label{lines} 
There are 15 $K$-orbits of lines in $\PG(5,q)$ as described in Tables \ref{tableoflinesodd} and \ref{tableoflineseven}.
\end{Theorem}

    For $q$ odd, $K$-orbits of solids in $\PG(5,q)$ correspond to $K$-orbits of lines in $\PG(5,q)$, whereas this correspondence fails when $q$ is even. The $K$-orbits of solids for $q$ even were determined in \cite{solidsqeven}. We collect  representatives of solids in $\PG(5,q)$, $q$ odd, in Table \ref{solidsodd}. 

Hyperplanes of $\PG(5,q)$ correspond to conics of $\PG(2,q)$ through the Veronese map $\nu$. We make this correspondence explicit via the following bijection $\delta$ between conics of $\PG(2,q)$ and hyperplanes of $\PG(5,q)$:
\[
\delta:  \cZ\Big(\sum_{0\leqslant i \leqslant j \leqslant 2}a_{ij}X_iX_j\Big)\mapsto\mathcal{Z}(a_{00}Y_0+a_{01}Y_1+a_{02}Y_2+a_{11}Y_3+a_{12}Y_4+a_{22}Y_5).
\]
In particular, we have four $K$-orbits of hyperplanes defined as follows: $\cH_{1}$, $\cH_{2,r}$ and $\cH_{2,i}$ denote the $K$-orbits of hyperplanes corresponding to the $\PGL(3,q)$-orbits of double lines, pairs of real lines and pairs of conjugate imaginary lines in $\PG(2,q)$, respectively, and $\cH_3$ denotes the $K$-orbit of hyperplanes corresponding to the $\PGL(3,q)$-orbit of non-singular conics in $\PG(2,q)$. Alternatively, $\mathcal{H}_1$ is the set of hyperplanes intersecting $\cV(\Fq)$ in a conic, $\mathcal{H}_{2r}$ is the set of hyperplanes intersecting $\cV(\Fq)$ in two conics, $\mathcal{H}_{2i}$ is the set of hyperplanes intersecting $\cV(\Fq)$ in a point, and $\mathcal{H}_3$ is the set of hyperplanes intersecting $\cV(\Fq)$ in a normal rational curve. The next lemma is a consequence of the correspondence between hyperplanes of $\PG(5,q)$ and conics of $\PG(2,q)$ outlined above.

\begin{Lemma}
    \ A hyperplane $H=\mathcal{Z}(f)\in\PG(5,q)$, where $f=a_{0}Y_0+a_{1}Y_1+a_{2}Y_2+a_{3}Y_3+a_{4}Y_4+a_{5}Y_5$ belongs to $\mathcal{H}_3$ if and only if its dual coordinates satisfy $$4a_{0}a_{3}a_{5}+a_{1}a_{2}a_{4}-a_{0}a_{4}^2-a_{3}a_{2}^2-a_{5}a_{1}^2\neq 0.$$

\end{Lemma}

Similarly as for the point-orbit distribution, sometimes the shorthand notation $OD_4(W)=[h_1, h_{2,r}, h_{2,i}, h_3]$ is used to represent the hyperplane-orbit distribution of the given subspace $W$ of $\PG(5,q)$, where $h_1=h_1(W)$ denotes the number of hyperplanes in $\cH_1$ incident with $W$ etc.

\section{Squabs of conics}\label{squabs}
In this section, for each squab ${\mathcal{S}}$ of conics in $\PG(2,q)$, we determine the number of different types of conics contained in $\mathcal S$.
Since a squab of conics is a linear system of co-dimension one, and the vector space ${\mathcal{F}}_2(2,q)$ has dimension 6, a squab $\mathcal S$ corresponds to a point $P_{\mathcal{S}}$ in the dual $\PG(5,q)$ of $\PG({\mathcal{F}}_2(2,q))$, which we identify with the ambient space of the Veronese variety $\cV(\bF_q)$ as described above. 
The {\it rank of a squab $\mathcal S$} is defined as the rank of $P_{\mathcal{S}}$ in $\PG(5,q)$ as defined above.
This gives a one-to-one correspondence between the number of different types of conics contained in the squab $\mathcal S$ and the hyperplane-orbit distribution $OD_4(P_{\mathcal{S}})$.
The main results of this section are Theorem \ref{thm:hyporbdist_pts_odd} for $q$ odd, and Theorem \ref{thm:hyporbdist_pts_even} for $q$ even, in which we determine $OD_4(P)$ for each point $P$ in $\PG(5,q)$. While the hyperplane-orbit distributions of the point-orbits depend on the parity of $q$, a number of interesting corollaries can be drawn from these results, which hold for both odd and even values of $q$. Recall that a squab of conics in $\PG(2,q)$ consists of $q^4+q^3+q^2+q+1$ conics.

\begin{Corollary}\label{cor:squabs}
(i) Every squab of conics in $\PG(2,q)$ contains exactly $q^4-q^2$
non-singular conics, unless the squab has rank three, in which case it contains $q^4$ non-singular conics.\\
(ii) If two squabs of conics in $\PG(2,q)$ have the same number of double lines and the same number of non-singular conics, then they are projectively equivalent.
\\
(iii) The number of double lines is a complete invariant for squabs of conics in $\PG(2,q)$ of rank two.\\
(iv) The number of non-singular conics is a complete invariant for squabs of conics in $\PG(2,q)$ not of rank two.\\
(v) Every squab of conics in $\PG(2,q)$ contains exactly $q+1$ double lines, unless the squab has rank two.
\end{Corollary}

The remainder of this section is dedicated to the proof of the main results concerning squabs of conics in $\PG(2,q)$.

\begin{Theorem}\label{thm:hyporbdist_pts_odd}
The hyperplane-orbit distributions of the point-orbits in $\PG(5,q)$, $q$ odd, are as in Table \ref{table_hypOD_points}.
\begin{table}[!htbp]
\begin{center}
\small
\begin{tabular}[h]{l | l} 

$P^K$ & $OD_4(P)$ \\ \hline
\\
$\mathcal{P}_1$&$[q+1,q^3+\frac{3q^2+q}{2},\frac{q^2-q}{2},q^4-q^2]$\\
\\
$\mathcal{P}_{2,e}$& $[2q+1, \frac{q^3+3q^2}{2},\frac{q^3+q^2}{2}-q, q^4-q^2]$\\
\\
$\mathcal{P}_{2,i}$ &$[1,\frac{q^3+3q^2}{2}+q,\frac{q^3+q^2}{2},q^4-q^2]$\\
\\
$\mathcal{P}_3$ &$[q+1, \frac{q^3+2q^2+q}{2},\frac{q^3-q}{2}, q^4]$\\
\hline
 \end{tabular}
 \caption{\label{table_hypOD_points}Hyperplane-orbit distributions of the $K$-orbits of points in $\PG(5,q)$; $q$ odd.}
\end{center}
\end{table}
\end{Theorem}
\begin{proof}
Recall the sizes of the point orbits and the hyperplane orbits: $|\cP_1|=|\cH_1|=q^2+q+1$, $|\cP_{2,r}|=|\cH_{2,r}|=q(q+1)(q^2+q+1)/2$, $|\cP_{2,i}|=|\cH_{2,i}|=q(q-1)(q^2+q+1)/2$, and $|\cP_3|=|\cH_3|=q^5-q^2$.

\bigskip

(i) Let $P\in \cP_1$. There are $q+1$ hyperplanes in $\cH_1$ through $P$, corresponding to lines through $\nu^{-1}(P)$. There are $q^3+\frac{3q^2+q}{2}$ hyperplanes in $\cH_{2,r}$ through $P$, ${q+1 \choose 2}$ of them corresponding to pairs of distinct lines belonging to the pencil of lines through $\nu^{-1}(P)$, and $(q+1)q^2$ hyperplanes corresponding to pairs of distinct lines, exactly one of which containing $\nu^{-1}(P)$. To count the number of hyperplanes in $\cH_{2,i}$, consider the embedding of $\PG(2,q)$ in $\PG(2,q^2)$ as a Baer subplane. There are $q^2-q$ lines in $\PG(2,q^2)$ through $\nu^{-1}(P)$, which do not belong to $\PG(2,q)$, and each of them is paired up with its conjugate. It follows that there are $\frac{q^2-q}{2}$ hyperplanes in $\cH_{2,i}$ through $P$. The remaining $q^4-q^2$ hyperplanes through $P$ belong to $\cH_3$. Note that $h_{2,r}+h_{2,i}=q^3+2q^2$.

\bigskip

(ii) Let $P\in \cP_{2,e}$. Let $R_1$ and $R_2$ denote the points of tangency of the tangent lines of $\cC(P)$ through $P$. 

Each line $L$ through $\nu^{-1}(R_i)$ corresponds to a hyperplane $H_L \in \cH_1$ which contains the tangent plane of $
\cV(\bF_q)$ through $R_i$, and therefore contains $P$. There are $2q+1$ such hyperplanes $H_L$, and each hyperplane of $\cH_1$ through $P$ is of this type, i.e. $h_1=2q+1$. 

To determine $h_{2,r}$, first observe that any hyperplane $H\in\cH_{2,r}$ containing $\cC(P)$ also contains $P$, and that there are $q^2+q$ such hyperplanes. Next we count hyperplanes in $\cH_{2,r}$ through $P$ which are not containing $\cC(P)$. Let $\alpha$ denote the number of such hyperplanes through $P$, and note that $\alpha$ is independent of the choice of $P\in \cP_{2,e}$.
 Suppose $H$ is such a hyperplane, containing the two conic planes $\pi_i$, $i=1,2$, with $R=\pi_1\cap \pi_2$. Let $\pi$ be any conic plane different from $\pi_1$ and $\pi_2$, and let $L_\pi=\pi\cap H$. If $R\in L_\pi$ then $L_\pi$ is tangent to $\cV(\bF_q)$ in $\pi$. In this case $L_\pi$ contains $q$ points of $\cP_{2,e}$. If $R\notin L_\pi$ then $L$ is a 2-secant to $\cV(\bF_q)$ in $\pi$, and therefore contains $\frac{q-1}{2}$ points of $\cP_{2,e}$. Counting flags $(P,H)$ with $P\in \cP_{2,e}$, $H\in \cH_{2,r}$, and $\cC(P)$ not contained in $H$, gives
$$
|\cH_{2,r}| \left ( (q-1)q+q^2\left ( \frac{q-1}{2}\right ) \right )=|\cP_{2,e}| ~ \alpha,
$$
since two conic planes meet in a point of $\cV(\bF_q)$, and there are $q-1$ conic planes $\pi$, different from $\pi_1$ and $\pi_2$, for which $R\in L_\pi$.
This implies that 
$$\alpha=\frac{q^3+q^2}{2}-q,
\mbox{ and therefore, } h_{2,r}=\alpha+q^2+q=\frac{q^3+3q^2}{2}.$$

Next we determine $h_{2,i}$. Since a hyperplane $H\in \cH_{2,i}$ meets each conic plane in a line $L$, the number of points of $\cP_{2,e}$ in $H$ is equal to 
$$
\beta=(q+1)q+q^2\left (\frac{q+1}{2}\right )=(q+1)\left ( \frac{q^2+2q}{2}\right )
$$
where the first term corresponds to the points contained in the $q+1$ conic planes through $R=H\cap \cV(\bF_q)$, and the second term corresponds to the points contained in the $q^2$ remaining conic planes.
Counting flags $(P,H)$ with $P\in \cP_{2,e}$ and $H\in \cH_{2,i}$, gives
$|\cH_{2,i}| \beta=|\cP_{2,e}| ~ h_{2,i}$,
which implies 
$$
h_{2,i}=\frac{|\cH_{2,i}|}{|\cP_{2,e}|}(q+1)\left ( \frac{q^2+2q}{2}\right )=(q-1)\left ( \frac{q^2+2q}{2}\right )
= \frac{q^3+q^2}{2}-q.
$$

\bigskip

(iii) Let $P\in \cP_{2,i}$. We claim that the hyperplane-orbit distribution of $P$ is 
$$[h_1, h_{2,r}, h_{2,i}, h_3]=[1,\frac{q^3+3q^2}{2}+q,\frac{q^3+q^2}{2},q^4-q^2].
$$
Since $P$ is contained in the hyperplane of $\cH_1$ defined by $\cC(P)$ and two hyperplanes of $\cH_1$ meet in a solid which does not contain any points of $\cP_{2,i}$, it is clear that $h_1=1$.

To determine $h_{2,r}$, first observe that there are $q^2+q$ hyperplanes of $\cH_{2,r}$ containing $\cC(P)$. Let $\alpha$ denote the number of hyperplanes through $P$, which are not containing $\cC(P)$, and suppose $H$ is such a hyperplane, containing the two conic planes $\pi_i$, $i=1,2$, with $R=\pi_1\cap \pi_2$. Let $\pi$ be any conic plane different from $\pi_1$ and $\pi_2$, and let $L_\pi=\pi\cap H$. If $R\in L_\pi$ then $L_\pi$ is tangent to $\cV(\bF_q)$ in $\pi$. In this case $L_\pi$ contains no points of $\cP_{2,i}$. If $R\notin L_\pi$ then $L$ contains $\frac{q-1}{2}$ points of $\cP_{2,i}$. Counting flags $(P,H)$ with $P\in \cP_{2,i}$, $H\in \cH_{2,r}$, and $\cC(P)$ not contained in $H$, gives
$$
|\cH_{2,r}| \left (q^2\left ( \frac{q-1}{2}\right ) \right )=|\cP_{2,i}| ~ \alpha,
$$
since two conic planes meet in a point of $\cV(\bF_q)$, and there are $q-1$ conic planes $\pi$, different from $\pi_1$ and $\pi_2$, for which $R\in L_\pi$.
This implies that 
$$\alpha=\frac{q^3+q^2}{2},
\mbox{ and therefore, } h_{2,r}=\alpha+q^2+q=\frac{q^3+3q^2}{2}+q.$$

To prove that $h_{2,i}=\frac{q^3+q^2}{2}$, first note that the number of points of $\cP_{2,i}$ in a hyperplane $H$ of $\cH_{2,i}$ is equal to $\beta=\frac{q^3+q^2}{2}$. To see this, consider the intersection of $H$ with the $q^2+q+1$ conic planes of $\cV(\bF_q)$: $q+1$ of these conic planes are passing through the unique point of $H\cap \cV(\bF_q)$ and these conic planes do not contribute to $\beta$; while each of the other $q^2$ conic planes contributes $(q+1)/2$ points to $\beta$. The fact that $h_{2,i}=\beta$ now easily follows by counting flags $(P,H)$ where $P\in \cP_{2,i}$ and $H\in \cH_{2,i}$, since $|\cP_{2,i}|=|\cH_{2,i}|$.

Once again we obtain $h_3=q^4+q^3+q^2+q+1-h_1-h_{2,r}-h_{2,i}=q^4-q^2$.

\bigskip

(iv) Let $P\in \cP_{3}$. The number of hyperplanes of $\cH_1$ through $P$ is equal to the number of points of $\cV(\bF_q)$ in a hyperplane $H\in \cH_3$ which is $q+1$, since $H\cap \cV(\bF_q)$ is a normal rational curve. The values of $h_{2,r}$ and $h_{2,i}$ can easily be determined by counting flags $(P,H)$ with $H\in \cH_{2,r}$ and $H\in \cH_{2,i}$, respectively, where we use the fact (which we just proved) that each point of rank two is contained in exactly $q^4-q^2$ hyperplanes of $\cH_3$.
This gives
$$
|\cH_{2,r}| (q^4-q^2)=|\cP_{3}| h_{2,r}, \mbox{ and, } |\cH_{2,i}| (q^4-q^2)=|\cP_{3}| h_{2,i},
$$
which implies $h_{2,r}=\frac{q^3+2q^2+q}{2},$ and $h_{2,i}=\frac{q^3-q}{2}$. In this case one obtains $h_3=q^4$.
\end{proof}

\begin{Theorem}\label{thm:hyporbdist_pts_even}
The hyperplane-orbit distributions of the point-orbits in $\PG(5,q)$, $q$ even, are as in Table \ref{table_hypOD_points_even}.
\begin{table}[!htbp]
\begin{center}
\small
\begin{tabular}[h]{l | l} 
$P^K$ & $OD_4(P)$ \\ \hline
\\
$\mathcal{P}_1$&$[q+1,q^3+\frac{3q^2+q}{2},\frac{q^2-q}{2},q^4-q^2]$\\
\\
$\mathcal{P}_{2,n}$& $[q^2+q+1, \frac{q^3+2q^2+q}{2},\frac{q^3-q}{2}, q^4-q^2]$\\
\\
$\mathcal{P}_{2,s}$ &$[q+1,\frac{q^3+3q^2+q}{2},\frac{q^3+q^2-q}{2},q^4-q^2]$\\
\\
$\mathcal{P}_3$ &$[q+1, \frac{q^3+2q^2+q}{2},\frac{q^3-q}{2}, q^4]$\\
\hline
 \end{tabular}
 \caption{\label{table_hypOD_points_even}Hyperplane-orbit distributions of the $K$-orbits of points in $\PG(5,q)$; $q$ even.}
\end{center}
\end{table}
\end{Theorem}
\begin{proof}
Parts (i) and (iv) in proof of the Theorem \ref{thm:hyporbdist_pts_odd} for the hyperplane-orbit distribution of the point-orbits $\cP_1$ and $\cP_3$ is independent of the parity of $q$. 

Recall that $|\cP_{2,n}|=(q^2+q+1)$, and $|\cP_{2,s}|=(q^2-1)(q^2+q+1)$.

(i) Let $P\in \cP_{2,n}$. Then $\cP$ is the nucleus of a unique conic $\cC(P)$ and each conic plane $\pi$ meets $\cC(P)$ in a point $R_\pi$. The unique hyperplane $H_\pi\in \cH_1$ containing $\pi$, contains the tangent plane of $\cV(\bF_q)$ at $R_\pi$ and therefore contains $P$. Therefore $h_1=q^2+q+1$. In fact this proves that the nucleus plane is contained in each hyperplane of the orbit $\cH_1$. 

There are $q^2+q$ hyperplanes $H\in\cH_{2,r}$ containing $\cC(P)$, and the number $\alpha$ of hyperplanes in $\cH_{2,r}$ through $P$ which are not containing $\cC(P)$, can be obtained as follows. 
Suppose $H\in \cH_{2,r}$ is a hyperplane, containing the two conic planes $\pi_i$, $i=1,2$, with $R=\pi_1\cap \pi_2$. Let $\pi$ be any conic plane different from $\pi_1$ and $\pi_2$, and let $L_\pi=\pi\cap H$. If $R\in L_\pi$ then $L_\pi$ is tangent to $\cV(\bF_q)$ in $\pi$. In this case $L_\pi$ contains exactly one point of $\cP_{2,n}$, namely the nucleus of the conic $\pi\cap\cV(\bF_q)$. If $R\notin L_\pi$ then $L$ is a 2-secant to $\cV(\bF_q)$ in $\pi$ and does not meet the nucleus plane of $\cV(\bF_q)$. Counting flags $(P,H)$ with $P\in \cP_{2,n}$, $H\in \cH_{2,r}$, and $\cC(P)$ not contained in $H$, gives
$$
|\cH_{2,r}| (q-1)=|\cP_{2,n}| ~ \alpha,
$$
which implies
$$\alpha=\frac{q^3-q}{2}, 
\mbox{ and hence, } h_{2,r}=\alpha+q^2+q=\frac{q^3+2q^2+q}{2}.$$

Since the number of points of $\cP_{2,n}$ in a hyperplane $H\in \cH_{2,i}$ is equal to 
$\beta=q+1$ (namely the nuclei of the $q+1$ conics through the point $H\cap \cV(\bF_q)$),
counting flags $(P,H)$ with $P\in \cP_{2,n}$ and $H\in \cH_{2,i}$, gives
$$
h_{2,i}=\frac{|\cH_{2,i}|}{|\cP_{2,n}|}(q+1)= \frac{q^3-q}{2},
$$
this determines $h_3=q^4-q^2$.

(ii) Let $P\in \cP_{2,s}$, and let $R$ be the unique point of $\cV(\bF_q)$ on the tangent line of $\cC(P)$ through $P$. Then each of the $q+1$ hyperplanes $H\in \cH_1$ containing $R$ contains $P$. Also, no other hyperplane $H\in \cH_1$ contains $P$, since any hyperplane $H\in \cH_1$ not through $R$ meets $\cC(P)$ in a point $Q$, such that the line $L$ through $P$ and $Q$ is a secant to $\cC(P)$, and if $H$ would contain $P$, then $H$ would contain two distinct points of $\cC(P)$, implying that $\cC(P)\subset H$. Therefore $h_1=q+1$.

As before, the number of hyperplanes $H\in \cH_{2,r}$ through $P$ is equal to $q^2+q+\alpha$ where $\alpha$ is the number of hyperplanes of $\cH_{2,r}$ through $P$, which are not containing $\cC(P)$. Let $H$ be such a hyperplane, and let $\pi_1$ and $\pi_2$ be the two conic planes in $H$. Let $R=\pi_1\cap \pi_2$. A conic plane $\pi\notin \{\pi_1,\pi_2\}$ meets $H$ in a line $L_\pi$ which is tangent to $\cC(\pi)$ if $R\in L_\pi$ and secant to $\cC(\pi)$ otherwise. So in both case the line $L_\pi$ contains $q-1$ points of $\cP_{2,s}$. This implies
$$
|\cH_{2,r}| \left ( (q-1)(q^2+q-1) \right )=|\cP_{2,s}| ~ \alpha,
$$
which gives 
$$
h_{2,r}=\alpha+q^2+q=\frac{q^3+3q^2+q}{2}.
$$
To prove that $h_{2,i}=\frac{q^3+q^2-q}{2}$, first note that the number of points of $\cP_{2,s}$ in a hyperplane $H$ of $\cH_{2,i}$ is equal to $\beta=(q^2+q-1)(q+1)$, since each conic plane meets $H$ in a line $L$ which contains $q-1$ points of $\cP_{2,s}$ if $L$ contains the point $H\cap\cV(\bF_q)$, $q+1$ points of $\cP_{2,s}$ if $L$ does not contain the point $H\cap\cV(\bF_q)$, and no point of $H\cap \cP_{2,s}$ is contained in two distinct conic planes. Therefore we obtain
$$
h_{2,i}=\frac{|\cH_{2,i}|}{|\cP_{2,s}|}(q^2+q-1)(q+1)= \frac{q^3+q^2-q}{2}.
$$
Finally, this gives 
$$h_3=q^4+q^3+q^2+q+1-\left ( q+1 + \frac{q^3+3q^2+q}{2} + \frac{q^3+q^2-q}{2}\right )=q^4-q^2,
$$
which completes the proof.
\end{proof}

Combining both theorems for $q$ odd and $q$ even, we obtain the following.

\begin{Corollary}
    Every point of the secant variety of $\cV(\Fq)$ lies in $q^4-q^2$ hyperplanes of type $\cH_3$.
\end{Corollary}

\section{Webs of conics}\label{webs}

In this section, for each web ${\mathcal{W}}$ of conics in $\PG(2,q)$, we determine the number of different types of conics contained in $\mathcal W$. A web $\mathcal W$ corresponds to a line $\ell_{\mathcal{W}}$ in the dual $\PG(5,q)$ of $\PG({\mathcal{F}}_2(2,q))$, which we identify with the ambient space of the Veronese variety $\cV(\bF_q)$, as before. Consequently, the number of different types of conics contained in the web $\mathcal W$ corresponds to the hyperplane-orbit distribution $OD_4(\ell_{\mathcal{W}})$. The $K$-orbits of lines of $\PG(5,q)$ were  classified in \cite{lines}, and the notation $\ell_i$ is used for a line of type $o_i$ (see Tables \ref{tableoflinesodd} and \ref{tableoflineseven}), and $\mathcal{W}_i$ denotes a web of conics in $\PG(2,q)$ corresponding to $\ell_i$. We denote by $\mathcal{W}=(f_1,f_2,f_3,f_4)$ the web of conics $\cZ(af_1+bf_2+cf_3+df_4)$, $(a,b,c,d)\in \PG(3,q)$,
determined by $\mathcal{C}_i=\cZ(f_i)$ for $1\leq i\leq 4$, no three of which are contained in a net of conics.

\begin{Definition}
   \normalfont{  We define the \textit{cubic surface associated with 
   the web $\mathcal{W}=(f_1,f_2,f_3,f_4)$} as the zero locus  $\cZ(\Delta_f)$ in $\PG(3,q)$ of the discriminant $\Delta_f\in \bF_q[A,B,C,D]$ of the quadratic form $f=Af_1+Bf_2+Cf_3+Df_4$. We also refer to $\cZ(\Delta_f)$ as the}  \textit{cubic surface associated with the line $\ell_{\mathcal{W}}$}.

\end{Definition}

The main result of this section is Theorem \ref{webs1}, in which we prove that the  hyperplane-orbit distribution of lines gives a complete invariant for all but two of the $\PGL(3,q)$-equivalence classes of webs of conics in $\PG(2, q)$. In Theorem \ref{webs2}, we give a connection between the number of points of rank at most $2$ on a line $L$ of $\PG(5,q)$ and the cardinality of its associated cubic surface. A number of interesting corollaries can be drawn from these results. Particularly, we conclude in Corollary \ref{aux5} that lines of $\PG(5,q)$ that have the same point-orbit distribution, also share the same hyperplane-orbit distribution. We end this section by concluding that a line $L$ in $\PG(5, q)$ having $q + i$ points of rank $3$ lies in $q^3 + iq = q(q^2 + i)$
hyperplanes in $\cH_3$. This observation implies that points of rank at most $2$ on a line of $\PG(5, q)$ entirely characterize the singularity
of the corresponding web of conics in $\PG(2, q)$.\\

Note that, as a consequence of the existence of a polarity of $\PG(5, q)$, $q$ odd, that maps the set of conic planes of $\cV(\Fq)$ onto the set of tangent planes of $\cV(\Fq)$, the hyperplane-orbit distributions of lines in $\PG(5,q)$ correspond to point-orbit distributions of solids of $\PG(5,q)$ for $q$ odd.

\begin{Lemma}
The hyperplane-orbit distribution of a line in $o_5$ is $[1,2q^2+q,0,q^3-q^2]$. 
\end{Lemma}
\begin{proof}
Let $\ell_5$ be a line of type $o_5$. The point-orbit distribution of $\ell_5$ is given by $[2,\frac{q-1}{2},\frac{q-1}{2},0]$ for $q$ odd and $[2,0,q-1,0]$ for $q$ even (see Tables \ref{tableoflinesodd} and \ref{tableoflineseven}). Let $R_1=\nu(r_1)$ and $R_2=\nu(r_2)$ denote the two rank-$1$ points on $\ell_5$. Then $\ell_5$ is contained in a unique hyperplane in $\cH_1$, namely $\nu(\langle r_1, r_2\rangle)$. Additionally, there are $q^2+q$ pairs of real lines $\ell\cup\ell'$ with $\ell=\langle r_1, r_2\rangle$ such that the span of their images under $\nu$ contains $\ell_5$. Moreover, there are $q^2$ pairs of real lines $\ell\cup\ell'$ in $\PG(2,q)$ where $r_1\in \ell$ and $r_2 \in \ell'$, with $\langle r_1, r_2\rangle \notin \{\ell,\ell'\}$. Given that $|\ell_5 \cap \cV(\Fq)|=2$, it follows that no hyperplane in $\cH_{2,i}$  contains $\ell_5$. Hence, $h_1(\ell_5)=1$, $h_{2r}(\ell_5)=2q^2+q$ and $h_3(\ell_5)=q^3-q^2$.
\end{proof}

\begin{Lemma}
The hyperplane-orbit distribution of a line in $o_6$ is $[q+1, \frac{3q^2+q}{2},\frac{q^2-q}{2}, q^3-q^2]$.
\end{Lemma}
\begin{proof}

Consider the representative $\ell_6$ of the line-orbit $o_6$ 
as described in Tables \ref{tableoflinesodd} and \ref{tableoflineseven}. The associated web of conics is defined by $\mathcal{W}_6=(X_0X_2, X_1^2, X_1X_2, X_2^2)$. The singular conics
in $\mathcal{W}_6$  are the zero locus of $aX_0X_2+ bX_1^2+c X_1X_2+ dX_2^2$ with $a^2b=0$. If $a=b=0$, $\mathcal{W}_6$ contains the double line $X_2^2=0$ and $q$ pairs of real lines defined by $X_2(X_1+dX_2)=0$. In the case of $a=1$ and $b=0$, $\mathcal{W}_6$ has $q^2$ pairs of real lines defined by $X_2(X_0+cX_1+dX_2)=0$. Now, assume that $a=0$ and $b=1$. If $c=0$, singular conics of $\mathcal{W}_6$ reduce to $X_1^2+dX_2^2=0$, defining $q$ double lines if $q$ is even, and one double line, $\frac{q-1}{2}$ pairs of real lines, and $\frac{q-1}{2}$ pairs of imaginary lines if $q$ is odd. This depends on whether $d$ is zero, a non-zero square, or a non-square respectively. When $c\neq 0$, the quadratic $X_1^2+cX_1X_2+dX_2^2=0$ defines 
$\frac{q^2-q}{2}$ pairs of real lines and $\frac{q^2-q}{2}$ pairs of imaginary lines if $q$ is even depending on $\operatorname{Tr}(c^{-2}d)$ being zero or one respectively. For odd $q$, it results in $q-1$ double lines, $\frac{(q-1)^2}{2}$ pairs of real lines, and $\frac{(q-1)^2}{2}$ pairs of imaginary lines depending on $c^2-4d$ being zero, a non-zero square, or a non-square respectively. Therefore, $\mathcal{W}_6$ has a total of $q+1$ double lines, $\frac{3q^2+q}{2}$ pairs of real lines, $\frac{q^2-q}{2}$ pairs of imaginary lines, and $q^3-q^2$ non-singular conics.
\end{proof}

\begin{Lemma}
The hyperplane-orbit distribution of a line in $o_{8,1}$ is $[1,q^2+\frac{3}{2}q,\frac{q}{2},q^3-q]$ for $q$ even, and $[2, q^2+\frac{3q-1}{2},\frac{q-1}{2}, q^3-q]$ for $q$ odd.
\end{Lemma}
\begin{proof}

Consider the representative $\ell_{8,1}$ of the line-orbit $o_{8,1}$ 
as described in Tables \ref{tableoflinesodd} and \ref{tableoflineseven}. The associated web of conics is denoted by $\mathcal{W}_{8,1}=(X_0X_1, X_0X_2, X_1X_2, X_1^2+X_2^2)$.
The singular conics in $\mathcal{W}_{8,1}$ are the zero locus of
$aX_0X_1 +b X_0X_2 +cX_1X_2+d(X_1^2+X_2^2)$
with $abc-d(a^2+b^2)=0$. 

For $q$ even, first assume that $a=b$. Consequently, $ac=0$. For $a=c=0$, one obtains a double line. When $a\neq0$ and $c=0$, the forms $(X_1+X_2)(X_0+d(X_1+X_2))=0$, $d\in \bF_q$, define $q$ pairs of real lines. If $a=0$ and $c\neq 0$, the reduced forms $X_1X_2+d(X_1^2+X_2^2)$ give $\frac{q}{2}$ pairs of real lines and $\frac{q}{2}$ pairs of imaginary lines.  Now, assume $a\neq b$. The above condition for singular conics then leads to $d=\frac{abc}{b^2+a^2}$. This results in $q^2$ pairs of real lines of the form: $(aX_1+bX_2)(cbX_1+acX_2+(b^2+a^2)X_0)$. Consequently, $\mathcal{W}_{8,1}$ consists of a unique double line, along with $q^2+\frac{3}{2}q$ pairs of real lines, $\frac{q}{2}$ pairs of imaginary lines, and $q^3-q$ non-singular conics when $q$ is even. 

For $q$ odd, singular conics in the web $\mathcal{W}_{8,1}$ correspond to points on the cubic surface 
$${\mathcal{X}}=\cZ(2ABD-C(A^2+B^2)),$$ 
where the cubic form in $\bF_q[A,B,C,D]$ is obtained from the determinant of the matrix 
$$
M(a,b,c,d)=\begin{bmatrix}\cdot&a&b\\a&c&d\\ b&d&c\end{bmatrix},
$$
representing the hyperplanes through $\ell_{8,1}$. We count $\bF_q$-rational points in ${\mathcal{X}}$. First we count the points of ${\mathcal{X}}$ in $\cZ(AB)$, which intersects ${\mathcal{X}}$ in the union of the three concurrent lines $\cZ(A,B)$, $\cZ(B,C)$ and $\cZ(A,C)$. This gives $3q+1$ points on $\mathcal X$ in $\cZ(AB)$. Outside $\cZ(AB)$, there are $q(q-1)$ points on $\mathcal X$ parameterized by $(1,b,c,\frac{c(1+b^2)}{2b})$, $b\neq 0$. Therefore, the total number of $\bF_q$-rational points on $\mathcal X$ is $q^2+2q+1$. The double lines in $\mathcal{W}_{8,1}$ correspond to points with coordinates $(a,b,c,d)$ on $\mathcal X$ such that the matrix 
$
M(a,b,c,d)$
has rank one. This gives $2$ points with coordinates $(0,0,1,\pm1)$. We use \cite[Lemma 3.7]{nets} to determine $h_{2,r}(\ell_{8,1})$ and $h_{2,i}(\ell_{8,1})$. Let $m_{11}$, $m_{22}$, and $m_{33}$ denote the three $2\times 2$ principal minors $m_{11}$, $m_{22}$, and $m_{33}$ of $M(a,b,c,d)$, then the point $(a,b,c,d)$ belongs to the orbit $\cP_{2,e}$ if and only if all of $-m_{11}$, $-m_{22}$, and $-m_{33}$ are squares in $\bF_q$, and at least one of them is non-zero. Here
$-m_{11}=d^2-c^2$, $-m_{22}=b^2$, and $-m_{33}=a^2$. It follows that all rank-$2$ points on $\mathcal X\cap (\cZ(A,C)\cup \cZ(B,C)) $ are exterior. Similarly, as $\frac{c^2(1+b^2)^2}{4b^2}-c^2$ is always a square, it follows that the $q^2-q$ rank-2 points on $\mathcal X\setminus \cZ(AB)$ are all exterior. On $\mathcal X\cap \cZ(A,B)$, exterior points correspond to pairs $(c,d)\in \Fq^2\setminus\{(0,0)\}$ such that $d^2-c^2\in \square_q$. This gives $\frac{q-1}{2}$ exterior points, $\frac{q-1}{2}$ interior points and $2$ rank-$1$ points on $\mathcal X\cap \cZ(A,B)$. Therefore,  $\mathcal{W}_{8,1}$ features $2$ double lines, $q^2+\frac{3q-1}{2}$ pairs of real lines, $\frac{q-1}{2}$ pairs of imaginary lines, and $q^3-q$ non-singular conics over finite fields of odd characteristic.
\end{proof}

\begin{Lemma}
The hyperplane-orbit distribution of a line in $o_{8,2}$ is $[0,q^2+\frac{3q+1}{2},\frac{q+1}{2},q^3-q]$. Note that these types of lines are only defined for $q$ odd.
\end{Lemma}
\begin{proof}

Consider the representative $\ell_{8,2}$ of the line-orbit $o_{8,2}$ 
as described in Table \ref{tableoflinesodd}. The associated web of conics, is denoted by $\mathcal{W}_{8,2}=(X_0X_1,X_0X_2,X_1X_2,\delta X_1^2+X_2^2)$, where $\delta$ is a non-square. In particular, singular conics in the web $\mathcal{W}_{8,2}$ correspond to points on the cubic surface 
$${\mathcal{X}}=\cZ(2ABD-C(A^2+\delta B^2)).$$ 
Similarly, as in the proof of the case $o_{8,1}$, one easily verifies that the total number of $\bF_q$-rational points on $\mathcal X$ is $q^2+2q+1$. There are no double lines in $\mathcal{W}_{8,2}$ as  the matrix 
$$
M(a,b,c,d)=\begin{bmatrix}\cdot&a&b\\a&\delta c&d\\ b&d&c\end{bmatrix},$$
has no rank one points. Let $m_{11}$, $m_{22}$, and $m_{33}$ denote the three $2\times 2$ principal minors of $M(a,b,c,d)$, then the point $(a,b,c,d)$ belongs to the orbit $\cP_{2,e}$ if and only if all of $-m_{11}$, $-m_{22}$, and $-m_{33}$ are squares and at least one of them is non-zero. Here
$-m_{11}=d^2-\delta c^2$, $-m_{22}=b^2$, and $-m_{33}=a^2$. It follows that all rank-$2$ points on $\mathcal X\cap (\cZ(A,C)\cup \cZ(B,C)) $ are exterior. Similarly, as $\frac{c^2(1+\delta b^2)^2}{4b^2}-\delta c^2$ is always a square, it follows that the $q^2-q$ rank-2 points on $\mathcal X\setminus \cZ(AB)$ are all exterior. On $\mathcal X\cap \cZ(A,B)$, exterior points on $\mathcal X$ correspond to pairs $(c,d)\in \Fq^2\setminus\{(0,0)\}$ such that $d^2-\delta c^2\in \square_q$. This gives $\frac{q+1}{2}$ exterior points and $\frac{q+1}{2}$ interior points on $\mathcal X\cap \cZ(A,B)$. Therefore,  $\mathcal{W}_{8,2}$  lacks double lines, yet includes $q^2+\frac{3q+1}{2}$ pairs of real lines, $\frac{q+1}{2}$ pairs of imaginary lines, and $q^3-q$ non-singular conics.
\end{proof}

\begin{Lemma}
The hyperplane-orbit distribution of a line in $o_{8,3}$ is $[q+1,q^2+q,0,q^3-q]$. Note that these types of lines are only defined for $q$ even.
\end{Lemma}
\begin{proof}
Consider the representative $\ell_{8,3}$ of the line-orbit $o_{8,3}$ 
as described in Table \ref{tableoflineseven}. Its associated web of conics is defined by $\mathcal{W}_{8,3}=(X_0X_1,X_0X_2,X_1^2,X_2^2)$. The singular conics of $\mathcal{W}_{8,3}$ are the zero locus of $aX_0X_1+bX_0X_2+cX_1^2+dX_2^2$, satisfying $cb^2+da^2=0$. For the case when $a=0$, we find $bc=0$, leading to $q+1$ double lines and $q$ pairs of real lines defined by $cX_1^2+dX_2^2=0$ and $X_2(X_0+dX_2)=0$. When $a \neq 0$, we have $d=cb^2$, resulting in $q^2$ pairs of real lines of the form: $(X_1+bX_2)(X_0+cX_1+bcX_2)=0$. Consequently, $\mathcal{W}_{8,3}$ consists of $q+1$ double lines, $q^2+q$ pairs of real lines, and $q^3-q$ non-singular conics.
\end{proof}

\begin{Lemma}
The hyperplane-orbit distribution of a line in $o_{9}$ is $[1,q^2+q,0,q^3]$.
\end{Lemma}
\begin{proof}
Consider the representative $\ell_9$ of the line-orbit $o_9$ 
as described in Tables \ref{tableoflinesodd} and \ref{tableoflineseven}. Its associated web of conics is represented by $\mathcal{W}_{9}=(X_0X_1,X_0X_2-X_1^2,X_1X_2,X_2^2)$. The singular conics of $\mathcal{W}_{9}$ are the zero locus of $aX_0X_1+b(X_0X_2-X_1^2)+cX_1X_2+dX_2^2$ with $b^3-da^2+abc=0$. In the case where $a=0$, we find $b=0$, leading to one double line and $q$ pairs of real lines defined by $X_2(cX_1+dX_2)$. If $a \neq 0$, we may assume without loss of generality that $d=b^3+bc$, yielding $q^2$ pairs of real lines of the form: $(X_1+bX_2)(X_0+cX_2-bX_1+b^2X_2)$. Therefore, $\mathcal{W}_{9}$ comprises a unique double line, $q^2+q$ pairs of real lines, and $q^3$ non-singular conics.
\end{proof}

\begin{Lemma}
The hyperplane-orbit distribution of a line in $o_{10}$ is $[1,q^2+q,q^2,q^3-q^2]$.
\end{Lemma}
\begin{proof}
Consider the representative $\ell_{10}$ of the line-orbit $o_{10}$ 
as described in Tables \ref{tableoflinesodd} and \ref{tableoflineseven}, with the constraint $v_0\lambda^2+uv_0\lambda-1 \neq 0$ for all $\lambda\in \Fq$. The associated web of conics, denoted as $\mathcal{W}_{10}$, is given by $\mathcal{W}_{10}=(v_0^{-1}X_0^2+uX_0X_1-X_1^2,X_0X_2,X_1X_2,X_2^2)$. The singular conics of $\mathcal{W}_{10}$ are the zero locus of
$$a(v_0^{-1}X_0^2+uX_0X_1-X_1^2)+bX_0X_2+cX_1X_2+dX_2^2$$ with $-4v^{-1}a^2d+uabc-v^{-1}ac^2+ab^2-u^2a^2d=0$. In the case where $a=0$, $\mathcal{W}_{10}$ comprises a unique double line and $q^2+q$ pairs of real lines defined by $X_2(bX_0+cX_1+dX_2)=0$. For $a \neq 0$, $\mathcal{W}_{10}$ exhibits $q^2$ pairs of imaginary lines. In summary, $\mathcal{W}_{10}$ consists of a unique double line, $q^2+q$ pairs of real lines, $q^2$ pairs of imaginary lines, and $q^3-q^2$ non-singular conics.
\end{proof}

\begin{Lemma}
The hyperplane-orbit distribution of a line in $o_{12,1}$ is $[q+2,q^2+\frac{q-1}{2}, q^2-\frac{q+1}{2}, q^3-q^2]$ for $q$ odd, and $[q^2+q+1,\frac{q^2+q}{2},\frac{q^2-q}{2},q^3-q^2]$ for $q$ even.
\end{Lemma}
\begin{proof}
Consider the representative $\ell_{12,1}$ of the line-orbit $o_{12,1}$ 
as described in Tables \ref{tableoflinesodd} and \ref{tableoflineseven}. The associated web of conics, denoted as $\mathcal{W}_{12,1}$, is represented by $\mathcal{W}_{12,1}=(X_0^2, X_0X_2, X_1^2, X_2^2)$. The singular conics of $\mathcal{W}_{12,1}$ are the zero locus of $aX_0^2+ bX_0X_2+cX_1^2+dX_2^2$, with $4acd-cb^2=0$. 

Let $q$ be odd. If $a=c=0$, $\mathcal{W}_{12,1}$ contains a double line and $q$ pairs of real lines defined by $bX_0X_2+dX_2^2=0$. For $a=0$ and $c\neq 0$, where $b=0$, $\mathcal{W}_{12,1}$ has a double line, $\frac{q-1}{2}$ pairs of real lines, and $\frac{q-1}{2}$ pairs of imaginary lines, depending on whether $d$ is zero, a non-zero square, or a non-square respectively. If $a\neq 0$ and $c=0$, then $\mathcal{W}_{12,1}$ has $q$ double lines, $\frac{q(q-1)}{2}$ pairs of real lines, and $\frac{q(q-1)}{2}$ pairs of imaginary lines defined by $X_0^2+bX_0X_2+dX_2^2=0$. For $ac\neq 0$, assuming $c=1$ and $d=\frac{b^2}{4a}$, $\mathcal{W}_{12,1}$ exhibits $\frac{q(q-1)}{2}$ pairs of real lines and $\frac{q(q-1)}{2}$ pairs of imaginary lines defined by $aX_0^2+bX_0X_2+X_1^2+\frac{b^2}{4a}X_2^2=0$. In total, over finite fields of odd characteristic, $\mathcal{W}_{12,1}$ comprises $q+2$ double lines, $q^2+\frac{q-1}{2}$ pairs of real lines, $q^2-\frac{q+1}{2}$ pairs of imaginary lines, and $q^3-q^2$ non-singular conics. 

Assume now that $q$ is even. If $b=0$, $\mathcal{W}_{12,1}$ contains $q^2+q+1$ double lines. Otherwise, when $c=0$, the forms $aX_0^2+X_0X_2+dX_2^2$, define $\frac{q^2+q}{2}$ pairs of real lines and $\frac{q^2-q}{2}$ pairs of imaginary lines. In summary, when is $q$ is even, $\mathcal{W}_{12,1}$ comprises a total of $q^2+q+1$ double lines, $\frac{q^2+q}{2}$ pairs of real lines, $\frac{q^2-q}{2}$ pairs of imaginary lines, and $q^3-q^2$ non-singular conics.
\end{proof}

\begin{Lemma}
The hyperplane-orbit distribution of a line in $o_{12,3}$ is $[q+1,q^2+\frac{q}{2}, q^2-\frac{q}{2}, q^3-q^2]$. Note that these types of lines are only defined for $q$ even.
\end{Lemma}
\begin{proof}
Consider the representative $\ell_{12,3}$ of the line-orbit $o_{12,3}$ 
as described in Table \ref{tableoflineseven}. The associated web of conics, denoted as $\mathcal{W}_{12,3}$, is represented by $\mathcal{W}_{12,3}=(X_0^2,X_0X_2,X_0X_1+X_1X_2+X_1^2,X_2^2)$. The singular conics of $\mathcal{W}_{12,3}$ are the zero locus of $$aX_0^2+bX_0X_2+c(X_0X_1+X_1X_2+X_1^2)+dX_2^2$$ with 
$c((a+b+d)c+b^2)=0$. For $c=0$, the quadratic $aX_0^2+bX_0X_2+dX_2^2=0$ defines $q+1$ double lines, $\frac{q^2+q}{2}$ pairs of real lines, and $\frac{q^2-q}{2}$ pairs of imaginary lines. If $c=1$, then $a=b+b^2+d$, and the forms in $\mathcal{W}_{12,3}$ reduce to 
$$(bX_0+X_1)^2+(X_0+X_2)(bX_0+X_1)+d(X_0+X_2)^2.$$ 
This yields $\frac{q^2}{2}$ pairs of real lines and $\frac{q^2}{2}$ pairs of imaginary lines. Consequently, $\mathcal{W}_{12,3}$ contains $q+1$ double lines, $q^2+\frac{q}{2}$ pairs of real lines, $q^2-\frac{q}{2}$ pairs of imaginary lines, and $q^3-q^2$ non-singular conics.
\end{proof}

\begin{Lemma}\label{w_{13,1}}
The hyperplane-orbit distribution of a line in $o_{13,1}$ is $[3,\frac{q^2+3q-2}{2},\frac{q^2+q-2}{2},q^3-q]$ for $q$ odd, and $[q+1,\frac{q^2}{2}+q,\frac{q^2}{2},q^3-q]$ for $q$ even.
\end{Lemma}

\begin{proof}
Consider the representative $\ell_{13,1}$ of the line-orbit $o_{13,1}$ 
as described in Tables \ref{tableoflinesodd} and \ref{tableoflineseven}. The associated web of conics, denoted as $\mathcal{W}_{13,1}$, is given by $\mathcal{W}_{13,1}=(X_0^2, X_0X_2, X_1^2+ X_2^2,X_1X_2)$. The singular conics of $\mathcal{W}_{13,1}$ are the zero locus of $$aX_0^2+bX_0X_2+c (X_1^2+ X_2^2)+dX_1X_2 $$ with $a(4c^2-d^2)-cb^2=0$. 

Let $q$ be even. If $c=0$, then $ad=0$, leading to a double line and $2q$ pairs of real lines defined by $aX_0^2+bX_0X_2$ and $bX_0X_2+X_1X_2$. If $c\neq0$ and $d=0$, then $\mathcal{W}_{13,1}$ contains $q$ double lines defined by $aX_0^2+X_1^2+X_2^2$. Otherwise, the forms of the singular conics of $\mathcal{W}_{13,1}$ reduce to 
$$\left(\frac{b}{d}X_0+X_1\right)^2+X_2^2+dX_2\left(\frac{b}{d}X_0+X_1\right),$$ 
yielding $q\left(\frac{q}{2}-1\right)$ pairs of real lines and $q\left(\frac{q}{2}\right)$ pairs of imaginary lines depending on whether $\operatorname{Tr}(d)$ is zero or not. In summary, $\mathcal{W}_{13,1}$ has $q+1$ double lines, $\frac{q^2}{2}+q$ pairs of real lines, $\frac{q^2}{2}$ pairs of imaginary lines, and $q^3-q$ non-singular conics over finite fields of even order.

If $q$ odd, singular conics in the web $\mathcal{W}_{13,1}$ correspond to points on the cubic surface 
$${\mathcal{X}}=\cZ(A(C^2-D^2)-CB^2),$$ 
where the cubic form is obtained from the determinant of the matrix 
$$
M(a,b,c,d)=\begin{bmatrix}a&\cdot&b\\\cdot&c&d\\ b&d&c\end{bmatrix},
$$
representing the hyperplanes through $\ell_{13,1}$.
We count $\bF_q$-rational points in ${\mathcal{X}}$. First we count the points of ${\mathcal{X}}$ in $\cZ(B)$, which intersects ${\mathcal{X}}$ in the union of the line $\cZ(A,B)$ and the conic $\cZ(B,C^2-D^2)$. This gives $3q$ points on $\mathcal X$ in $\cZ(B)$.
Next we count the points of ${\mathcal{X}}$ outside the plane $\cZ(B)$. The points of ${\mathcal{X}}$ in the affine part $\cZ(C)\setminus \cZ(B)$ of the plane $\cZ(C)$ are the $2q-1$ points on the affine lines $\cZ(A,C)\setminus \cZ(B)$ and $\cZ(C,D)\setminus \cZ(B)$. Outside the union of $\cZ(B)$, $\cZ(C)$ and $\cZ(C^2-D^2)$, there is a unique point of $\mathcal X$ on each line through $(1,0,0,0)$ and a point $(0,b,c,d)$ in $\cZ(A)\setminus \cZ(BC(C^2-D^2))$. This gives another $q^2-3q+2$ points. Therefore, the total number of $\bF_q$-rational points on $\mathcal X$ is $3q+2q-1+q^2-3q+2=q^2+2q+1$. It follows that $h_3(\ell_{13,1})=q^3-q$. \par
The double lines in $\mathcal{W}_{13,1}$ correspond to points with coordinates $(a,b,c,d)$ on $\mathcal X$ such that the matrix 
$M(a,b,c,d)$ has rank one. There are three such points on $\mathcal X$, parameterized by $(0,0,1,1)$, $(0,0,1,-1)$ and $(1,0,0,0)$.

To determine $h_{2,r}(\ell_{13,1})$ and $h_{2,i}(\ell_{13,1})$, we use \cite[Lemma 3.7]{nets}. If $m_{11}$, $m_{22}$, and $m_{33}$ denote the three $2\times 2$ principal minors of $M(a,b,c,d)$, then the point $(a,b,c,d)$ parameterizes a point belonging to the orbit $\cP_{2,e}$ (corresponding to a hyperplane in $\cH_{2,r}$) if and only if all of $-m_{11}$, $-m_{22}$, and $-m_{33}$ are squares and at least one of them is non-zero. Here
$-m_{11}=d^2-c^2$, $-m_{22}=b^2-ac$, and $-m_{33}=-ac$. 
In $\mathcal{X}\cap \cZ(C^2-D^2)$, there are $2q-1$ points parameterizing hyperplanes in $\cH_{2,r}$: $q-1$ points parameterized by $(a,0,1,\pm1)$ such that $-a \in \square_q\setminus\{0\}$, and  $q$ points parameterized by $(a,1,0,0)$, $a \in \Fq$. 
In $\mathcal{X}\setminus \cZ(C^2-D^2)$, points on $\mathcal X$ are parameterized by $(\frac{b^2c}{c^2-d^2},b,c,d)$, and the expressions for $-m_{22}$, and $-m_{33}$ become
$$
-m_{22}=b^2-\frac{b^2c^2}{c^2-d^2}=\frac{-b^2d^2}{c^2-d^2} \mbox{, and } -m_{33}=\frac{-b^2c^2}{c^2-d^2}.
$$
So such a point corresponds to a hyperplane in $\cH_{2,r}$ if and only if $d^2-c^2$ is a nonzero square. The number of such hyperplanes is therefore equal to the cardinality of the set 
$$\{(b,c,d)~:~b\in \bF_q, (c,d)\in \PG(1,q), d^2-c^2\in \square_q\setminus \{0\}\},$$
which is equal to $q(q-1)/2$. In total this gives 
$$h_{2,r}(\ell_{13,1})=2q-1+\frac{q(q-1)}{2}=\frac{q^2+3q-2}{2}.
$$
The above implies that 
$h_{2,i}(\ell_{13,1})=q^2+2q+1-3-\frac{q^2+3q-2}{2}=\frac{q^2+q-2}{2}$.
\end{proof}

\begin{Lemma}
The hyperplane-orbit distribution of a line in $o_{13,2}$ is $[1,\frac{q^2+3q}{2},\frac{q^2+q}{2},q^3-q]$. Note that these types of lines are only defined for $q$ odd. 
\end{Lemma}
\begin{proof}

Consider the representative $\ell_{13,2}$ of the line-orbit $o_{13,2}$ 
as described in Table \ref{tableoflinesodd}.  The web of conics associated with $\ell_{13,2}$ is represented by $\mathcal{W}_{13,2}=(X_0^2, X_0X_2, \delta X_1^2+X_2^2,X_1X_2)$, where $\delta$ is a non-square. The proof is very similar to the proof for $o_{13,1}$, $q$ odd. Singular conics in the web $\mathcal{W}_{13,2}$ correspond to points on the cubic surface 
$${\mathcal{X}}=\cZ(A(\delta C^2-D^2)-\delta CB^2),$$ 
where the cubic form is obtained from the determinant of the matrix 
$$
M(a,b,c,d)=\begin{bmatrix}a&\cdot&b\\\cdot&\delta c&d\\ b&d&c\end{bmatrix},
$$
representing the hyperplanes through $\ell_{13,2}$. As in the proof for $o_{13,1}$, $q$ odd, computing the total number of $\bF_q$-rational points on $\mathcal X$ one obtains $q+2+2q-1+q^2-q=q^2+2q+1$. It follows that $h_3(\ell_{13,2})=q^3-q$. \par
The double lines in $\mathcal{W}_{13,2}$ correspond to points with coordinates $(a,b,c,d)$ on $\mathcal X$ such that the matrix 
$M(a,b,c,d)$ has rank one. This gives a unique such point on $\mathcal X$ parameterized by $(1,0,0,0)$, since $\delta $ is a nonsquare.

To determine $h_{2,r}(\ell_{13,2})$ and $h_{2,i}(\ell_{13,2})$, we use \cite[Lemma 3.7]{nets}. If $m_{11}$, $m_{22}$, and $m_{33}$ denote the three $2\times 2$ principal minors of $M(a,b,c,d)$, then the point $(a,b,c,d)$ parameterizes a point belonging to the orbit $\cP_{2,e}$ (corresponding to a hyperplane in $\cH_{2,r}$) if and only if all of $-m_{11}$, $-m_{22}$, and $-m_{33}$ are squares and at least one of them is non-zero. Here
$-m_{11}=d^2-\delta c^2$, $-m_{22}=b^2-ac$, and $-m_{33}=-\delta ac$. 
In $\mathcal{X}\cap \cZ(\delta C^2-D^2)=\cZ(C,D)$, all points on the line $\cZ(C,D)$ except $(1,0,0,0)$ parameterize hyperplanes in $\cH_{2,r}$.
In $\mathcal{X}\setminus \cZ(\delta C^2-D^2)$, points on $\mathcal X$ are parameterized by $(\frac{b^2c}{c^2-d^2},b,c,d)$, and the expressions for $-m_{22}$, and $-m_{33}$ become
$$
-m_{22}=b^2-\frac{\delta b^2c^2}{\delta c^2-d^2}=\frac{-b^2d^2}{\delta c^2-d^2} \mbox{, and } -m_{33}=\frac{-\delta^2 b^2c^2}{\delta c^2-d^2}.
$$
So such a point corresponds to a hyperplane in $\cH_{2,r}$ if and only if $d^2-\delta c^2$ is a nonzero square. The number of such hyperplanes is therefore equal to the cardinality of the set 
$$\{(b,c,d)~:~b\in \bF_q, (c,d)\in \PG(1,q), d^2-\delta c^2\in \square_q\setminus \{0\}\},$$
which is equal to $q(q+1)/2$. In total this gives 
$$h_{2,r}(\ell_{13,2})=q+\frac{q(q+1)}{2}=\frac{q^2+3q}{2}.
$$
The above implies that 
$h_{2,i}(\ell_{13,2})=q^2+2q+1-1-\frac{q^2+3q}{2}=\frac{q^2+q}{2}$.
\end{proof}

\begin{Lemma}
The hyperplane-orbit distribution of a line in $o_{13,3}$ is $[1,\frac{q^2+3q}{2},\frac{q^2+q}{2},q^3-q]$. Note that these types of lines are only defined for $q$ even.
\end{Lemma}
\begin{proof}
Consider the representative $\ell_{13,3}$ of the line-orbit $o_{13,3}$ 
as described in Table \ref{tableoflineseven}. The associated web of conics, denoted as $\mathcal{W}_{13,3}$, is defined by $\mathcal{W}_{13,3}=(X_0^2, X_0X_2, X_0X_1+X_1^2+ X_2^2,X_1X_2)$. The singular conics in $\mathcal{W}_{13,3}$ are the zero locus of 
$$aX_0^2+b X_0X_2+c( X_0X_1+X_1^2+ X_2^2)+dX_1X_2$$ with 
$ad^2+cb^2+c^3+bcd=0$. If $c=0$, a discussion similar to that in Lemma \ref{w_{13,1}} reveals a double line and $2q$ pairs of real lines. When $c\neq 0$ and $d=0$, $\mathcal{W}_{13,3}$ contains $\frac{q}{2}$ pairs of real lines and $\frac{q}{2}$ pairs of imaginary lines defined by $aX_0^2+X_0(X_1+X_2)+(X_1+X_2)^2$. Otherwise, assuming without loss of generality that $c=1$ and $a=\frac{b^2+bd+1}{d^2}$, $\mathcal{W}_{13,3}$ comprises $q(\frac{q}{2}-1)$ pairs of real lines and $q(\frac{q}{2})$ pairs of imaginary lines represented by $(d^{-1}X_0+X_2)^2+d^{-2}(bX_0+dX_1)^2+(d^{-1}X_0+X_2)(bX_0+dX_1)$. In total, $\mathcal{W}_{13,3}$ contains a distinctive double line, $\frac{q^2+3q}{2}$ pairs of real lines, $\frac{q^2+q}{2}$ pairs of imaginary lines, and $q^3-q$ non-singular conics.
\end{proof}

\begin{Remark}\label{o_141}
For $q$ odd, the $K$-orbit \( o_{14,1} \) correspond to the \( K \)-orbit \( \Omega_{14,1} \) which contains four rank-$1$ points, parameterized by $(1,1,1,1)$, $(1,1,-1,-1)$, $(1,-1,-1,1)$, and $(1,-1,1,-1)$. Their pre-images under \( \nu \) compose the frame
$\{(1,1,1),(1,1,-1), (1,-1,1),(-1,1,1)\}$.
It follows that an alternative representative for $\Omega_{14,1}$ is given by the solid
\[
M(a,b,c,d)=
\begin{bmatrix}
a & d & d \\
d & b & d \\
d & d & c \\
\end{bmatrix}
\]
which is the span of the images under \( \nu \) of the standard frame in \( \PG(2,q) \).
\end{Remark}

\begin{Lemma}
The hyperplane-orbit distribution of a line in $o_{14,1}$ is $[4,\frac{q^2-1}{2}+2q-1,\frac{q^2-1}{2}+q-1,q^3-2q]$ for $q$ odd, and $[1,\frac{q^2}{2}+2q,\frac{q^2}{2}+q,q^3-2q]$ for $q$ even.
\end{Lemma}
\begin{proof}
Consider the representative $\ell_{14,1}$ of the line-orbit $o_{14,1}$ 
as described in Tables \ref{tableoflinesodd} and \ref{tableoflineseven}. The associated web of conics, denoted as $\mathcal{W}_{14,1}$, is represented by $\mathcal{W}_{14,1}=(X_0X_1, X_0X_2, X_0^2+X_1^2+ X_2^2,X_1X_2)$. The singular conics in $\mathcal{W}_{14,1}$ are the zero locus of 
$$aX_0X_1+b X_0X_2+c (X_0^2+X_1^2+ X_2^2)+dX_1X_2$$ with 
$abd-c(a^2+b^2+d^2-4c^2)=0.$ 

Suppose $q$ is even. If $a=0$, then $c(b+d)=0$ and $\mathcal{W}_{14,1}$ contains the unique double line $\cZ((X_0+X_1+ X_2)^2)$, along with $\frac{3q}{2}$ pairs of real lines and $\frac{q}{2}$ pairs of imaginary lines defined by $X_2(bX_0+dX_1)$ and $c(X_0+X_1)^2+bX_2(X_0+X_1)+cX_2^2$. Otherwise, when $a \neq 0$, we distinguish between the cases where $d=1+b$ and $d\neq 1+b$. In the first case, $\mathcal{W}_{14,1}$ comprises $q$ pairs of real lines and $q$ pairs of imaginary lines defined by $c(X_0+X_2)^2+X_1(X_0+X_2)+cX_1^2$ and $c(X_1+X_2)^2+X_0(X_1+X_2)+cX_0^2$. In the latter case, we have $c=\frac{bd}{(1+b+d)^2}$, and similar to $\mathcal{W}_{10}$, the resulting quadratic defines $q(\frac{q-1}{2})$ pairs of real lines and $q(\frac{q-1}{2})$ pairs of imaginary lines. Therefore, $\mathcal{W}_{14,1}$ has a unique double line, $\frac{q^2}{2}+2q$ pairs of real lines, $\frac{q^2}{2}+q$ pairs of imaginary lines, and $q^3-2q$ non-singular conics over finite fields of even order. 

For $q$ odd, the double lines in $\mathcal{W}_{14,1}$ correspond to the four points with coordinates $(a,b,c,d)$ such that the matrix $M(a,b,c,d)$, as in Remark \ref{o_141}, has rank one.
 Pairs of lines in $\mathcal{W}_{14,1}$ correspond to rank-$2$ points $(a,b,c,d)$ on the cubic surface 
$$\cZ(A(BC-D^2)+D^2(2D-C-B)).$$ 
In particular, pairs of real lines correspond to exterior rank-$2$ points characterized by having $-m_{11}=d^2-bc$, $-m_{22}=d^2-ac$ and $-m_{33}=d^2-ab$ squares, with at least one nonzero.

The cubic surface  $\mathcal{X}$ intersects the plane $\cZ(D)$ in $3q-3$ rank-$2$ points parameterized by $(0,1,c,0)$, $(1,0,c,0)$ and $(1,b,0,0)$; $bc\neq 0$. In particular, exterior points in $\mathcal{X}\cap\cZ(D)$ are characterized by $-c\in \square_q\setminus\{0\}$ and $-b\in \square_q\setminus\{0\}$. This gives $3(\frac{q-1}{2})$ exterior and $3(\frac{q-1}{2})$ interior points in $\mathcal{X}\cap\cZ(D)$. In $\mathcal{X}\setminus \cZ(D)$, we start by counting points on $\cZ(BC-D^2)$. This gives $q-1$ rank-$2$ points parameterized by $(a,d,d,d)$; $d \neq 0$, with $\frac{q-1}{2}$ exterior and $\frac{q-1}{2}$ interior points characterized by having $d-a \in\square_q\setminus\{0\}$ and $d-a\not\in \square_q $, respectively. 
In $\mathcal{X}\setminus \cZ(D(BC-D^2))$, points on $\mathcal{X}$ are parameterized by $(\frac{d^2(2d-c-b)}{d^2-bc},b,c,d)$, and the expressions for $-m_{22}$, and $-m_{33}$ become
$$
-m_{22}=d^2-\frac{cd^2(2d-c-b)}{d^2-bc}=\frac{(d^2-c)^2}{d^2-bc} \mbox{, and } -m_{33}=d^2-\frac{bd^2(2d-c-b)}{d^2-bc}=\frac{(d^2-bd)^2}{d^2-bc} .
$$
So such a point corresponds to a hyperplane in $\cH_{2,r}$ if and only if $d^2- bc$ is a nonzero square. The number of such hyperplanes is therefore equal to the cardinality of the set 
$$\{(b,c,d)~:~d\in \bF_q\setminus\{0\}, (b,c)\in \PG(1,q), d^2-bc\in \square_q\setminus \{0\}\},$$
which is equal to $(q-1)(q+1)/2$. In total this gives 
$$h_{2,r}(\ell_{14,1})=2q-2+\frac{(q-1)(q+1)}{2}=\frac{q^2+4q-3}{2}.
$$
The above implies that 
$h_{2,i}(\ell_{14,1})=2q-2+q^2-q-\frac{(q^2-1)}{2}=\frac{q^2+2q-3}{2}$.
\end{proof}

\begin{Lemma}
The hyperplane-orbit distribution of a line in $o_{14,2}$ is $[0,\frac{q^2+1}{2}+2q,\frac{q^2+1}{2}+q,q^3-2q]$. Note that these types of lines are only defined for $q$ odd.
\end{Lemma}

\begin{proof}
Consider the representative $\ell_{14,2}$ of the line-orbit $o_{14,2}$ 
as described in Table \ref{tableoflinesodd}, where $\delta$ is a non-square.  The associated web of conics, denoted as $\mathcal{W}_{14,2}$, is represented by $\mathcal{W}_{14,2}=(X_0X_1,X_0X_2,\delta X_0^2+X_1^2+\delta X_2^2,X_1X_2)$. The singular conics in $\mathcal{W}_{14,2}$ are the zero locus of $$aX_0X_1+bX_0X_2+c(\delta X_0^2+X_1^2+\delta X_2^2)+dX_1X_2 $$ with 
$abd-c(\delta a^2+b^2+\delta d^2-4\delta^2c^2)=0.$
The point-orbit distribution of $\ell_{14,2}$ is $[0,1,2,q-2]$.
\par
Let $\{P\}=\ell_{14,2}\cap \cP_{2,e}$, and let $R_1$ and $R_2$ denote the two points of $\ell_{14,2}$ in $\cP_{2,i}$. First observe that $P$ corresponds to $(x,y)=(1,-1)$, and the unique hyperplane $H_{\cC(P)}\in \cH_1$ which contains $\cC(P)$ is equal to $\cZ(Y_3)$, and does not contain $\ell_{14,2}$. Next, observe that, by Theorem \ref{thm:hyporbdist_pts_odd}, for each $i\in \{1,2\}$, there is a unique hyperplane of $\cH_1$ through $R_i$, and this hyperplane does not contain $\ell_{14,2}$. This implies that $h_1(\ell_{14,2})=0$.

Next we count flags $(\pi,H_\pi)$ where $\pi$ is a conic plane not containing any of the conics in $\{\cC(P),\cC(R_1),\cC(R_2)\}$, and $H_\pi \in \cH_{2,r}$ is the hyperplane $H_\pi=\langle \ell_{14,2},\pi\rangle$. Each of the three solids $\langle \ell_{14,2},\cC(P)\rangle$, $\langle \ell_{14,2},\cC(R_1)\rangle$, and $\langle \ell_{14,2},\cC(R_2)\rangle$, is contained in $q+1$ hyperplanes of $\cH_{2,r}$, and there are three hyperplanes in $\cH_{2,r}$ which contain exactly two of these three solids. Therefore, there are $\alpha=3q-3$ hyperplanes of $\cH_{2,r}$ containing exactly one of the conics in $\{\cC(P),\cC(R_1),\cC(R_2)\}$. This gives the equation
$$q^2+q-2=3q-3 +2 \beta
$$
where $\beta$ is the number of hyperplanes $H\in \cH_{2,r}$ through $\ell_{14,2}$ not containing any of the conics $\{\cC(P),\cC(R_1),\cC(R_2)\}$. It follows that $\beta=(q^2+4q-5)/2$, and that the number of hyperplanes in $\cH_{2,r}$ through $\ell_{14,2}$ is equal to
$h_{2,r}(\ell_{14,2})=3+\alpha+\beta=(q^2+4q+1)/2$.

To determine $h_{2,i}(\ell_{14,2})$ we count incident pairs $(\pi,H_\pi)$ where $\pi$ is a tangent hyperplane to $\cV(\bF_q)$ and $H_\pi$ is a hyperplane containing $\langle \pi,\ell_{14,2}\rangle$. There are two tangent planes meeting $\ell_{14,2}$ (in the point $P$) and for each such tangent plane $\pi$, the solid $\langle \pi,\ell_{14,2}\rangle$ is contained in $q+1$ hyperplanes of $\cH_{2,r}\cup \cH_{2,i}$. Since $h_{2,r}(\ell_{14,2})=3+\alpha+\beta=(q^2+4q+1)/2$ this gives
$$
\frac{q^2+4q+1}{2}+h_{2,i}(\ell_{14,2})=2(q+1)+q^2+q-1,
$$
from which it follows that $h_{2,i}(\ell_{14,2})=(q^2+2q+1)/2$, and therefore $h_3(\ell_{14,2})=q^3-2q$.
\end{proof}

\begin{Lemma}\label{o15,1}
The hyperplane-orbit distribution of a line in $o_{15,1}$ is $[2,\frac{q^2-1}{2}+q,\frac{q^2-1}{2},q^3]$ for $q$ odd, and $[1,\frac{q^2}{2}+q,\frac{q^2}{2},q^3]$   for $q$ even.
\end{Lemma}

\begin{proof}
Consider the representative $\ell_{15,1}$ of the line-orbit $o_{15,1}$ 
as described in Tables \ref{tableoflinesodd} and \ref{tableoflineseven}. Here, $v_1\lambda^2+uv_1\lambda-1 \neq 0$ for all $\lambda\in \Fq$, and $-v_1$ is a nonzero square. The associated web of conics, denoted as $\mathcal{W}_{15,1}$, is given by $\mathcal{W}_{15,1}=(X_0X_2,X_1X_2,X_0X_1-X_2^2,v_1^{-1}X_0^2+uX_0X_1-X_1^2)$.

Suppose $q$ is odd. Let $\{P\}=\ell_{15,1}\cap \cP_{2,e}$, and let $R_1$ and $R_2$ denote the two points of tangency of the tangent lines to $\cC(P)$ through $P$. First observe that the unique hyperplane $H_{\cC(P)}\in \cH_1$ which contains $\cC(P)$ is equal to $\cZ(Y_5)$, and does not contain $\ell_{15,1}$. Let $S$ be the solid spanned by $\ell_{15,1}$ and $\cC(P)$. Each hyperplane $H_\pi\in \cH_1$ through $R_i$ (for $i\in \{1,2\}$) contains $P$ and meets $S$ in a plane through the line $\langle R_i,P\rangle$. This gives a one-to-one correspondence between the $q+1$ hyperplanes in $\cH_1$ through $R_i$ and the $q+1$ planes of the plane-pencil in $S$ through $\langle R_i,P\rangle$. Hence exactly one of these planes contains $\ell_{15,1}$. Since the hyperplane $H_{\cC(P)}\in \cH_1$ does not contain $\ell_{15,1}$, there are exactly two such hyperplanes of $\cH_1$ containing $\ell_{15,1}$: one through $R_1$ and one through $R_2$. Since $h_1(P)=2q+1$ (Theorem \ref{thm:hyporbdist_pts_odd}), no other hyperplane of $\cH_1$ contains $P$, and therefore $h_1(\ell_{15,1})=2$.

Next we count flags $(\pi,H_\pi)$ where $\pi$ is a conic plane not containing $\cC(P)$, and $H_\pi\in \cH_1\cup \cH_{2,r}$ is the hyperplane 
$H_\pi=\langle \ell_{15,1},\pi\rangle$. This gives
$$q^2+q=2+\alpha +2 \beta
$$
where $\alpha$ is the number of hyperplanes $H\in \cH_{2,r}$ through $\ell_{15,1}$ containing $\cC(P)$ and $\beta$ is the number of hyperplanes $H\in \cH_{2,r}$ not containing $\cC(P)$. Since there are $q+1$ hyperplanes through the solid $S=\langle \ell_{15,1},\cC(P)\rangle$, and none of these hyperplanes belongs to $\cH_1$ (see above), it follows that $\alpha=q+1$. This gives $\beta=(q^2-3)/2$ and therefore $h_{2,r}(\ell_{15,1})=\alpha+\beta=(q^2+2q-1)/2$.

By the first part of the proof, exactly two of the hyperplanes of $\cH_1$ contain $\langle \pi,\ell_{15,1}\rangle$, with $\pi$ a tangent plane of $\cV(\bF_q)$, and the total number of tangent planes contained in these two hyperplanes is $2(q+1)$. Call this set of tangent planes $\mathcal T$. Also there are
$(q^2+2q-1)/2$ hyperplanes $H_\pi=\langle \pi,\ell_{15,1}\rangle$, with $\pi\notin \mathcal T$ a tangent plane of $\cV(\bF_q)$, which belong to $\cH_{2,r}$. 
Since a hyperplane in $\cH_1$ contains $q+1$ tangent planes of $\cV(\bF_q)$, and a hyperplane in $\cH_{2,r}\cup \cH_{2,i}$ contains exactly one tangent plane of $\cV(\bF_q)$, counting flags $(\rho,H_\pi)$, where $\rho$ and $\pi$ are tangent planes of $\cV(\bF_q)$, and $H_\pi$ is a hyperplane containing $\pi$ and $\ell_{15,1}$, one obtains
$$
2(q+1)+\frac{q^2+2q-1}{2}+h_{2,i}(\ell_{15,1})=2(q+1)+q^2+q-1,
$$
where the $2(q+1)$ on the right hand side arises from the fact that if $\pi$ is one of the two tangent planes which meet $\ell_{15,1}$, then the solid $\langle \ell_{15,1},\pi\rangle$ is contained in $q+1$ hyperplanes. This gives $h_{2,i}(\ell_{15,1})=(q^2-1)/2$. It follows that $h_3(\ell_{15,1})=q^3$.

\par
The proof is similar for $q$ is even. Again, let $P$ be the unique point of rank two on $\ell_{15,1}$, and recall that $P\in \cP_{2,s}$.
As in the case for $q$ odd, observe that the unique hyperplane $H_{\cC(P)}\in \cH_1$ which contains $\cC(P)$ is equal to $\cZ(Y_5)$, and does not contain $\ell_{15,1}$. However, in this case there is exactly one hyperplane of $\cH_1$ which contains $\ell_{15,1}$, which is one of the $q+1$ hyperplanes in $\cH_1$ containing a conic through the point of tangency of the unique tangent line to $\cC(P)$ through $P$. So $h_1(\ell_{15,1})=1$. A similar reasoning as in the case $q$ odd, gives
$$
h_{2,r}(\ell_{15,1})=\alpha+\beta, \mbox{ where } \alpha=q+1, \mbox{ and } 1+\alpha+2\beta=q^2+q.
$$
Therefore $h_{2,r}(\ell_{15,1})=(q^2+2q)/2$, and 
$$
(q+1)+\frac{q^2+2q}{2}+h_{2,i}(\ell_{15,1})=(q+1) + (q^2+q),
$$
which gives $h_{2,i}(\ell_{15,1})=q^2/2$. It follows that $h_3(\ell_{15,1})=q^3$.

\end{proof}

\begin{Remark}\label{o151}
Note that the proof of the hyperplane-orbit distribution of $o_{15,1}$ depends on the point-orbit distribution of $o_{15,1}$ and on the fact that the hyperplane $H\in \cH_1$ containing $\cC(P)$ does not contain the line $\ell_{15,1}$.
\end{Remark}

\begin{Lemma}
The hyperplane-orbit distribution of a line in $o_{15,2}$ is $[0,\frac{q^2+1}{2}+q,\frac{q^2+1}{2},q^3]$. Note that these types of lines are only defined for $q$ odd.
\end{Lemma}

\begin{proof}
Consider the representative $\ell_{15,2}$ of the line-orbit $o_{15,2}$ 
as described in Table \ref{tableoflinesodd}, where $v_2\lambda^2+uv_2\lambda-1 \neq 0$ for all $\lambda\in \Fq$ and   $-v_2$ is a non-square in $\bF_q$.
The associated web of conics, denoted as $\mathcal{W}_{15,2}$, is given by $\mathcal{W}_{15,2}=(X_0X_2,X_1X_2,X_0X_1-X_2^2,v_2^{-1}X_0^2+uX_0X_1-X_1^2)$. 

\par
Let $\{P\}=\ell_{15,2}\cap \cP_{2,i}$.
As in the previous case, the unique hyperplane $H_{\cC(P)}\in \cH_1$ which contains $\cC(P)$ is equal to $\cZ(Y_5)$, and does not contain $\ell_{15,2}$. Since (by Theorem \ref{thm:hyporbdist_pts_odd}) $h_1(P)=1$, there are no hyperplanes of $\cH_1$ containing $\ell_{15,2}$.

Next we count flags $(\pi,H_\pi)$ where $\pi$ is a conic plane not containing $\cC(P)$, and $H_\pi\in \cH_1\cup \cH_{2,r}$ is the hyperplane 
$H_\pi=\langle \ell_{15,2},\pi\rangle$. This gives
$$q^2+q=\alpha +2 \beta
$$
where $\alpha$ is the number of hyperplanes $H\in \cH_{2,r}$ through $\ell_{15,2}$ containing $\cC(P)$ and $\beta$ is the number of hyperplanes $H\in \cH_{2,r}$ not containing $\cC(P)$. Since there are $q+1$ hyperplanes through the solid $S=\langle \ell_{15,2},\cC(P)\rangle$, and none of these hyperplanes belongs to $\cH_1$ (see above), it follows that $\alpha=q+1$. This gives $\beta=(q^2-1)/2$ and therefore $h_{2,r}(\ell_{15,2})=\alpha+\beta=(q^2+2q+1)/2$.

So there are $(q^2+2q+1)/2$ hyperplanes $H_\pi=\langle \pi,\ell_{15,1}\rangle$, with $\pi$ a tangent plane of $\cV(\bF_q)$, which belong to $\cH_{2,r}$. 
Since a hyperplane in $\cH_{2,r}\cup \cH_{2,i}$ contains exactly one tangent plane of $\cV(\bF_q)$, counting flags $(\rho,H_\pi)$, where $\rho$ and $\pi$ are tangent planes of $\cV(\bF_q)$, and $H_\pi$ is a hyperplane containing $\pi$ and $\ell_{15,2}$, one obtains
$$
\frac{q^2+2q+1}{2}+h_{2,i}(\ell_{15,2})=q^2+q+1,
$$
since none of the tangent planes meets $\ell_{15,2}$.
This gives $h_{2,i}(\ell_{15,2})=(q^2+1)/2$. It follows that $h_3(\ell_{15,2})=q^3$.
\end{proof}

\begin{Lemma}\label{o16,1}
The hyperplane-orbit distribution of a line in $o_{16,1}$ is $[2,\frac{q^2-1}{2}+q,\frac{q^2-1}{2},q^3]$ for $q$ odd, and $[q+1,\frac{q^2+q}{2}, \frac{q^2-q}{2},q^3]$ for $q$ even.
\end{Lemma}
\begin{proof}
Consider the representative $\ell_{16,1}$ of the line-orbit $o_{16,1}$ 
as described in Tables \ref{tableoflinesodd} and \ref{tableoflineseven}. The associated web of conics, denoted as $\mathcal{W}_{16,1}$, is given by $\mathcal{W}_{16,1}=(X_0^2,X_0X_1,X_0X_2-X_1^2,X_2^2)$. The singular conics of $\mathcal{W}_{16,1}$ are the zero locus of $$aX_0^2+bX_0X_1+c(X_0X_2-X_1^2)+dX_2^2$$ satisfying $-4acd+c^3-db^2=0$. Let's first assume that $q$ is odd. If $cd=0$, singular conics of $\mathcal{W}_{16,1}$ reduce to $aX_0^2+X_2^2=0$ and $aX_0^2+bX_0X_1=0$, defining $2$ double lines, $q+\frac{q-1}{2}$ pairs of real lines, and $\frac{q-1}{2}$ pairs of imaginary lines. Otherwise, if we set $d=1$, then $a=\frac{c^3-b^2}{4c}$, and $\mathcal{W}_{16,1}$ comprises $\frac{q(q-1)}{2}$ pairs of real lines and $\frac{q(q-1)}{2}$ pairs of imaginary lines, as explained in $\mathcal{W}_{10}$. For $q$ even, $\mathcal{W}_{16,1}$ has $q+1$ double lines defined by $aX_0^2+dX_2^2=0$. If $b\neq0$, we can set $d=c^3$, leading to $\frac{q(q+1)}{2}$ pairs of real lines and $\frac{q(q-1)}{2}$ pairs of imaginary lines of the form $aX_0^2+X_0(X_1+cX_2)+c(X_1+cX_2)^2=0$.
In summary, $\mathcal{W}_{16,1}$ has a total of $2$ (resp. $q+1$) double lines, $\frac{q^2-1}{2}+q$ (resp. $\frac{q^2+q}{2}$) pairs of real lines, $\frac{q^2-1}{2}$ (resp. $\frac{q^2-q}{2}$) pairs of imaginary lines, and $q^3$ non-singular conics over finite fields of odd (resp. even) order.
\end{proof}

\begin{Remark}\label{o161}
The unique hyperplane $H_{\cC(P)}\in \cH_1$ which contains $\cC(P)$, where $\{P\}=\ell_{16,1}\cap \cP_{2,e}$, also contains $\ell_{16,1}$. 

\end{Remark}

\begin{Lemma}\label{o16,3}
The hyperplane-orbit distribution of a line in $o_{16,3}$ is $[1,\frac{q^2}{2}+q,\frac{q^2}{2},q^3]$. Note that these types of lines are only defined for $q$ even.
\end{Lemma}
\begin{proof}
Consider the representative $\ell_{16,3}$ of the line-orbit $o_{16,3}$ 
as described in Table \ref{tableoflineseven}. The associated web of conics is defined as $\mathcal{W}_{16,3}=(X_0^2,X_0X_1,X_0X_2+X_1^2,X_1X_2+X_2^2)$. Singular conics in $\mathcal{W}_{16,3}$ are the zero locus of $$aX_0^2+bX_0X_1+c(X_0X_2+X_1^2)+d(X_1X_2+X_2^2)$$ with  $ad^2+c^3+db^2+bcd=0$. If $d=0$, $\mathcal{W}_{16,3}$ comprises a double line and $q$ pairs of real lines defined by the equation $aX_0^2+bX_0X_1=0$. Otherwise, assuming $a=c^3+b^2+bc$, the web comprises $\frac{q^2}{2}$ pairs of real lines and $\frac{q^2}{2}$ pairs of imaginary lines, represented by solutions to the equation $c(c^2X_0^2+X_1^2)+(cX_0+X_1)(bX_0+X_2)+(bX_0+X_2)^2=0$. In summary, $\mathcal{W}_{16,3}$ contains one double line, $\frac{q^2}{2}+q$ pairs of real lines, $\frac{q^2}{2}$ pairs of imaginary lines, and $q^3$ non-singular conics.

\end{proof}

\begin{Remark}\label{o163}
The unique hyperplane $H_{\cC(P)}\in \cH_1$ which contains $\cC(P)$, where $\{P\}=\ell_{16,3}\cap \cP_{2,s}$, also contains $\ell_{16,3}$. 
\end{Remark}

The proofs of the hyperplane-orbit distributions of the lines with point-orbit distribution $[0,1,0,q]$ for $q$ odd, and $[0,0,1,q]$ for $q$ even, imply the following lemma, giving an interesting new geometric invariant for the corresponding $K$-orbits on lines, which was not known before. This should be compared to the previously known method from \cite[Lemma 1]{T233paper}, to distinguish between the different $K$-orbits of such lines.

\begin{Lemma}\label{lineswithsameOD}
    For $q$ odd, a line $L$ of $\PG(5,q)$ with $OD_1(L)=[0,1,0,q]$ belongs to $o_{16,1}$ if and only if the unique hyperplane $H_{\cC(P)}\in \cH_1$ which contains $\cC(P)$, where $\{P\}=L\cap \cP_{2,e}$, also contains $L$. Otherwise, $L\in o_{15,1}$.  For $q$ even, a line $L$ of $\PG(5,q)$ with $OD_1(L)=[0,0,1,q]$ belongs to $o_{16,3}$ if and only if the unique hyperplane $H_{\cC(P)}\in \cH_1$ which contains $\cC(P)$, where $\{P\}=L\cap \cP_{2,s}$, also contains $L$. Otherwise, $L\in o_{15,1}$.

\end{Lemma}
\begin{proof}
The proof  follows from the proofs of the hyperplane-orbit distributions of such lines, see Remark \ref{o151}, \ref{o161} and \ref{o163}.
\end{proof}

\begin{Corollary}\label{aux5}
{ Lines of $\PG(5,q)$ that have the same point-orbit distribution, also share the same hyperplane-orbit distribution.}
\end{Corollary}
\begin{proof}
 The point-orbit distribution serves as a complete invariant for lines, except for the orbits $\{o_{15,1},o_{16,3}\}$ for $q$ even with a point-orbit distribution $[0,0,1,q]$ and $\{o_{15,1},o_{16,1}\}$ for $q$ odd with a point-orbit distribution $[0,1,0,q]$. Therefore, it suffices to show that lines in $\{o_{15,1},o_{16,3}\}$ (resp. $\{o_{15,1},o_{16,1}\}$) share the same hyperplane-orbit distributions for $q$ even (resp. odd), which  follows from Lemmas \ref{o15,1}, \ref{o16,1} and \ref{o16,3}. 
\end{proof}

\begin{Lemma}
The hyperplane-orbit distribution of a line in $o_{17}$ is $[1,\frac{q^2+q}{2},\frac{q^2-q}{2},q^3+q]$.
\end{Lemma}

\begin{proof}
Each conic plane $\pi$ of $\cV(\bF_q)$ determines a hyperplane $H_\pi=\langle \pi,\ell_{17}\rangle$, which belongs to $\cH_1\cup\cH_{2,r}$. Since there are $q^2+q+1$ such hyperplanes $H_\pi$, counting flags $(\pi,H_\pi)$ we have that $h_1+2h_{2,r}=q^2+q+1$. This implies that $h_1\geq 1$ and odd. We show that $h_1=1$. Consider two hyperplanes $H_{\pi}, H_{\pi'} \in \cH_1$. If $H_{\pi}$ and $H_{\pi'}$ are distinct, then they intersect in a solid $S$ which contains the tangent plane $\kappa$ of $\cV(\bF_q)$ at the point $\pi\cap \pi'$. This  is a contradiction, since $\ell_{17}$ and $\kappa$ are both contained in the solid $S$, so they would have to meet. But $\ell_{17}$ is a constant rank 3 line, and $\kappa$ does not contain any points of rank 3. Therefore $H_{\pi}=H_{\pi'}$. This shows that $h_1\leq 1$, and by the above $h_1=1$, and $h_{2,r}=\frac{q^2+q}{2}$.

Similarly, each tangent plane $\pi$ of $\cV(\bF_q)$ determines a hyperplane $H_\pi=\langle \pi,\ell_{17}\rangle$.
Such a hyperplane belongs to $\cH_1\cup\cH_{2,r}\cup \cH_{2,i}$, since a tangent plane $\pi_P$ of $\cV(\bF_q)$ at a point $P$ intersects a hyperplane $H\in \cH_3$ in the tangent line at the point $P$ to the normal rational curve $H\cap \cV(\bF_q)$ of degree 4  in $H$.

By the first part of the proof, exactly one of the hyperplanes $H_\pi=\langle \pi,\ell_{17}\rangle$, with $\pi$ a tangent plane of $\cV(\bF_q)$, belongs to $\cH_1$, and $\frac{q^2+q}{2}$ belong to $\cH_{2,r}$. 
Since a hyperplane in $\cH_{2,r}\cup \cH_{2,i}$ contains exactly one tangent plane of $\cV(\bF_q)$, 
counting flags $(\rho,H_\pi)$, where $\rho$ and $\pi$ are tangent planes of $\cV(\bF_q)$, one obtains
$$
(q+1)+\frac{q^2+q}{2}+h_{2,i}=q^2+q+1.
$$
This shows that $h_{2,i}=\frac{q^2-q}{2}$, which completes the proof.
\end{proof}

\begin{Theorem}

The number of lines of type $o_i$ in a fixed $H\in\cH_j $ is  
$$\frac{|o_{i}|\times h_j}{|\cH_j|}.$$

\end{Theorem}
\begin{proof}
    The hyperplane-orbit distribution of a line of type $o_{i}$ implies that the number of lines of type $o_{i}$ in a hyperplane of $\cH_j$ can be determined by counting flags $(L,H)$ where $L$ is of type $o_{i}$ and $H\in \cH_j$:

$$|o_{i}|\times h_j=|\cH_j|\times (number\; of\; lines\; of \; type\; o_{i}\; in\; a\; fixed\; H\in\cH_j).$$

\end{proof}

\begin{Theorem}
The line-orbit distributions of the hyperplane-orbits in $\PG(5,q)$, $q$ odd, are as in Table \ref{OD2(H)odd}.

\begin{table}[h!]
\centering
\scriptsize{
\begin{tabular}{ccccc}
\toprule
\multicolumn{1}{c}{Orbits} & \multicolumn{1}{c}{\vspace{1pt}$\cH_1$} & \multicolumn{1}{c}{\vspace{1pt}$\cH_{2,r}$} & \multicolumn{1}{c}{\vspace{1pt}$\cH_{2,i}$} & \multicolumn{1}{c}{\vspace{1pt}$\cH_3$} \\
\midrule\vspace{3pt}
$o_5$                      & $\frac{1}{2}q(q+1)$                     & $2q^2+q$                                    & $0$                                         & $\frac{1}{2}q(q+1)$                     \\\vspace{3pt}
$o_6$                      & $(q+1)^2$                               & $3q+1$                                      & $q+1$                                       & $q+1$                                   \\\vspace{3pt}
$o_{8,1}$                  & $q^3(q+1)$                              & $\frac{1}{2}q^2(2q^2+3q-1)$                 & $\frac{1}{2}q^2(q+1)$                       & $\frac{1}{2}q^2(q+1)^2$                 \\\vspace{3pt}
$o_{8,2}$                  & $0$                                     & $\frac{q^2(q-1)(2q^2+3q+1)}{2(q+1)}$        & $\frac{1}{2}q^2(q+1)$                       & $\frac{1}{2}q^2(q^2-1)$                 \\\vspace{3pt}
$o_9$                      & $q(q^2-1) $                             & $2q(q^2-1)$                                 & $0$                                         & $q^2(q+1)$                              \\\vspace{3pt}
$o_{10}$                   & $\frac{1}{2}q(q-1)$                     & $q(q-1)$                                    & $q^2$                                       & $\frac{1}{2}q(q-1)$                     \\\vspace{3pt}
$o_{12,1}$                 & $q+2$                                   & $\frac{2q^2+q-1}{q(q+1)}$                   & $\frac{2q^2-q-1}{q(q-1)}$                   & $1$                                     \\\vspace{3pt}
$o_{13,1}$                 & $\frac{3}{2}q^3(q^2-1)$                 & $\frac{1}{2}q^2(q-1)(q^2+3q-2)$             & $\frac{1}{2}q^2(q+1)(q^2+q-2)$              & $\frac{1}{2}q^2(q+1)(q^2-1)$            \\\vspace{3pt}
$o_{13,2}$                 & $\frac{1}{2}q^3(q^2-1)$                 & $\frac{1}{2}q^3(q-1)(q+3)$                  & $\frac{1}{2}q^3(q+1)^2$                     & $\frac{1}{2}q^2(q+1)(q^2-1)$            \\\vspace{3pt}
$o_{14,1}$                 & $\frac{1}{6}q^3(q-1)(q^2-1)$            & $\frac{1}{24}q^2(q-1)^2(q^2+4q-3)$          & $\frac{1}{24}q^2(q^2-1)(q^2+2q-3)$          & $\frac{1}{24}q^2(q^2-1)(q^2-2)$         \\\vspace{3pt}
$o_{14,2}$                 & $0$                                     & $\frac{1}{8}q^2(q-1)^2(q^2+4q+1)$           & $\frac{1}{8}q^2(q^2-1)(q^2+2q+1)$           & $\frac{1}{8}q^2(q^2-1)(q^2-2)$          \\\vspace{3pt}
$o_{15,1}$                 & $\frac{1}{2}q^3(q-1)(q^2-1)$            & $\frac{1}{4}q^2(q-1)^2(q^2+2q-1)$           & $\frac{1}{4}q^2(q^2-1)^2$                   & $\frac{1}{4}q^4(q^2-1)$                 \\\vspace{3pt}
$o_{15,2}$                 & $0$                                     & $\frac{1}{4}q^2(q-1)^2(q^2+2q+1)$           & $\frac{1}{4}q^2(q^4-1)$                     & $\frac{1}{4}q^4(q^2-1)$                 \\\vspace{3pt}
$o_{16,1}$                 & $2q^2(q^2-1)$                           & $q(q-1)(q^2+2q-1)$                          & $q(q+1)(q^2-1)$                             & $q^3(q+1)$                              \\\vspace{3pt}
$o_{17}$                   & $\frac{1}{3}q^3(q-1)(q^2-1)$            & $\frac{1}{3}q^3(q-1)(q^2-1)$                & $\frac{1}{3}q^3(q-1)(q^2-1)$                & $\frac{1}{3}q^2(q^4-1)$  \\    
\bottomrule          
\end{tabular}

\caption{\label{OD2(H)odd}Line-orbits distributions of hyperplanes in $\PG(5,q)$, $q$ odd.} }
\end{table}

\end{Theorem}

\begin{Theorem}
The line-orbit distributions of the hyperplane-orbits in $\PG(5,q)$, $q$ even, are as in Table \ref{OD2(H)even}.

\begin{table}[h!]
\centering
\scriptsize{
\begin{tabular}{ccccc}
\toprule
\multicolumn{1}{c}{Orbits} & \multicolumn{1}{c}{\vspace{1pt}$\cH_1$} & \multicolumn{1}{c}{\vspace{1pt}$\cH_{2,r}$} & \multicolumn{1}{c}{\vspace{1pt}$\cH_{2,i}$} & \multicolumn{1}{c}{\vspace{1pt}$\cH_3$} \\

\midrule\vspace{3pt}
$o_5$                      & $\frac{1}{2}q(q+1)$                     & $2q^2+q$                                    & $0$                                         & $\frac{1}{2}q(q+1)$                     \\\vspace{3pt}
$o_6$                      & $(q+1)^2$                               & $3q+1$                                      & $q+1$                                       & $q+1$                                   \\\vspace{3pt}
$o_{8,1}$                  & $q^2(q^2-1)$                            & $(2q+3)(q-1)q^2$                            & $q^2(q+1)$                                  & $q(q+1)(q^2-1)$                         \\\vspace{3pt}
$o_{8,3}$                  & $q^2(q+1)$                              & $2q^2$                                      & $0$                                         & $q(q+1)$                                \\\vspace{3pt}
$o_9$                      & $q(q^2-1) $                             & $2q(q^2-1)$                                 & $0$                                         & $q^2(q+1)$                              \\\vspace{3pt}
$o_{10}$                   & $\frac{1}{2}q(q-1)$                     & $q(q-1)$                                    & $q^2$                                       & $\frac{1}{2}q(q-1)$                     \\\vspace{3pt}
$o_{12,1}$                 & $q^2+q+1$                               & $1$                                         & $1$                                         & $1$                                     \\\vspace{3pt}
$o_{12,3}$                 & $(q+1)(q^2-1)$                          & $(q-1)(2q+1)$                               & $(q+1)(2q-1)$                               & $q^2-1$                                 \\\vspace{3pt}
$o_{13,1}$                 & $q^2(q+1)(q^2-1)$                       & $q^2(q-1)(q+2)$                             & $q^3(q+1)$                                  & $q(q+1)(q^2-1)$                         \\\vspace{3pt}
$o_{13,3}$                 & $q^2(q-1)(q^2-1)$                       & $q^2(q+3)(q-1)^2$                           & $q^2(q+1)(q^2-1)$                           & $q(q^2-1)^2$                            \\\vspace{3pt}
$o_{14,1}$                 & $\frac{1}{6}q^3(q-1)(q^2-1)$            & $\frac{1}{6}q^3(q-1)^2(q+4)$                & $\frac{1}{6}q^3(q^2-1)(q+2)$                & $\frac{1}{6}q^2(q^2-1)(q^2-2)$          \\\vspace{3pt}
$o_{15,1}$                 & $\frac{1}{2}q^3(q-1)(q^2-1)$            & $\frac{1}{2}q^3(q-1)^2(q+2)$                & $\frac{1}{2}q^4(q^2-1)$                     & $\frac{1}{2}q^4(q^2-1)$                 \\\vspace{3pt}
$o_{16,1}$                 & $q(q+1)(q^2-1)$                         & $q(q^2-1)$                                  & $q(q^2-1)$                                  & $q^2(q+1)$                              \\\vspace{3pt}
$o_{16,3}$                 & $q(q-1)(q^2-1)$                         & $q(q-1)^2(q+2)$                             & $q^2(q^2-1)$                                & $q^2(q^2-1)$                            \\\vspace{3pt}
$o_{17}$                   & $\frac{1}{3}q^3(q-1)(q^2-1)$            & $\frac{1}{3}q^3(q-1)(q^2-1)$                & $\frac{1}{3}q^3(q-1)(q^2-1)$                & $\frac{1}{3}q^2(q^4-1)$                \\
\bottomrule
\end{tabular}
\caption{\label{OD2(H)even}Line-orbits distributions of hyperplanes in $\PG(5,q)$, $q$ even.} }
\end{table}

\end{Theorem}

We end this section by the following conclusions.

\begin{Theorem}\label{webs1}
The hyperplane-orbit distribution serves as an almost complete invariant for projectively inequivalent webs of conics in $\PG(2,q)$, with the exception of cases $\mathcal{W}_{15,1}$ and $\mathcal{W}_{16,1}$ for odd $q$, and $\mathcal{W}_{15,1}$ and $\mathcal{W}_{16,3}$ for even $q$. In these cases, we use Lemma \ref{lineswithsameOD} or \cite[Lemma 1]{T233paper} to differentiate between the respective orbits.
\end{Theorem}

\begin{Theorem}\label{webs2}

    A line $L$ in $\PG(5,q)$ intersects the secant variety in $i$ points if and only if its associated cubic surface has $q^2+ iq+1$ points, $i \in\{0,1,2,3,q+1\}.$

\end{Theorem}
\begin{proof}
   This is a consequence of the hyperplane-orbit distributions of lines in $\PG(5,q)$. 
\end{proof}

 The subsequent corollaries can be readily derived from Theorem  \ref{webs2}.

\begin{Corollary}\label{webs3}
 A line $L$ in $\PG(5,q)$ having $q+ i$ points of rank $3$ lies in $q^3+ iq=q(q^2+i)$ hyperplanes in $\cH_3$. In particular, $i\in\{-q,-2,-1,0,1\}$.
\end{Corollary}

\begin{Corollary}\label{rank3pts}
Points of rank at most $2$ on a line in $\PG(5,q)$ entirely characterize the singularity of the corresponding web of conics in $\PG(2,q)$.
\end{Corollary}

\section{Conclusion}
This paper utilizes the combinatorial and geometric properties of the quadric Veronesean in $\PG(5,q)$ to determine the number of different types of conics contained in squabs and webs of
conics in $\PG(2, q)$. These numbers serve as complete invariants to distinguish between the projective equivalence classes of these linear systems of conics. Our results contribute towards the solution of the longstanding open problem of classifying the projective equivalence classes of linear systems of conics over finite fields initiated by L. E. Dickson  \cite{dickson} in 1908.

\section*{Acknowledgements}

The first author acknowledges the support of {\it The Croatian Science Foundation}, project number 5713.

\newpage 
\begin{appendices}   

In this appendix we collect a summary of our main results in Tables \ref{tablewebsodd} and  \ref{tablewebseven}. We gather in Tables \ref{tableoflinesodd} and \ref{tableoflineseven} representatives of the $K$-orbits of lines and their associated point-orbit distributions from \cite{lines}. In Table \ref{solidsodd}, we list representatives of the $15$ $K$-orbits of solids in $\PG(5,q)$, $q$ odd.

\begin{table}[!htbp]
\begin{center}
\scriptsize
\begin{tabular}[!htbp]{ l l l} 
 \hline
$L^K$&Representatives& $OD_0(L)$\\ \hline
$o_5$&$\begin{bmatrix}x&\cdot&\cdot\\\cdot&y&\cdot\\ \cdot&\cdot&\cdot\end{bmatrix}$ &$[2,\frac{q-1}{2},\frac{q-1}{2},0]$\\
$o_6$&$\begin{bmatrix}x&y&\cdot\\y&\cdot&\cdot\\ \cdot&\cdot&\cdot\end{bmatrix}$  &$[1,q,0,0]$\\
$o_{8,1}$ &$\begin{bmatrix}x&\cdot&\cdot\\\cdot&y&\cdot\\ \cdot&\cdot&-y\end{bmatrix}$ &$[1,1,0,q-1]$\\
$o_{8,2}$ & $\begin{bmatrix}x&\cdot&\cdot\\\cdot&y&\cdot\\ \cdot&\cdot& -\delta y\end{bmatrix}$ &$[1,0,1,q-1]$\\
$o_9$ &$\begin{bmatrix}x&\cdot&y\\\cdot&y&\cdot\\ y&\cdot&\cdot\end{bmatrix}$ &$[1,0,0,q]$\\
$o_{10}$ & $\begin{bmatrix}v_0x&y&\cdot\\y&x+uy&\cdot\\ \cdot&\cdot&\cdot\end{bmatrix}$ &$[0,\frac{q+1}{2},\frac{q+1}{2},0]$\\
$o_{12,1}$&$\begin{bmatrix}\cdot&x&\cdot\\x&\cdot&y\\ \cdot&y&\cdot\end{bmatrix}$&$[0,q+1,0,0]$\\ 
$o_{13,1}$ &$\begin{bmatrix}\cdot&x&\cdot\\x&y&\cdot\\ \cdot&\cdot&-y\end{bmatrix}$ &$[0,2,0,q-1]$\\
$o_{13,2}$ &$\begin{bmatrix}\cdot&x&\cdot\\x&y&\cdot\\ \cdot&\cdot&-\delta y\end{bmatrix}$ &$[0,1,1,q-1]$\\
$o_{14,1}$ &$\begin{bmatrix}x&\cdot&\cdot\\\cdot&-(x+y)&\cdot\\ \cdot&\cdot&y\end{bmatrix}$ &$[0,3,0,q-2]$\\ 
$o_{14,2}$ & $\begin{bmatrix}x&\cdot&\cdot\\\cdot&-\delta(x+y)&\cdot\\ \cdot&\cdot&y\end{bmatrix}$&$[0,1,2,q-2]$\\ 
$o_{15,1}$ & $\begin{bmatrix}v_1y&x&\cdot\\x&ux+y&\cdot\\ \cdot&\cdot&x\end{bmatrix}$&$[0,1,0,q]$\\ 
$o_{15,2}$ & $\begin{bmatrix}v_2y&x&\cdot\\x&ux+y&\cdot\\ \cdot&\cdot&x\end{bmatrix}$&$[0,0,1,q]$\\
$o_{16,1}$ &$\begin{bmatrix}\cdot&\cdot&x\\\cdot&x&y\\ x&y&\cdot\end{bmatrix}$ &$[0,1,0,q]$\\  
$o_{17}$ & $\begin{bmatrix}\alpha^{-1}x&y&\cdot\\y&\beta y-\gamma x&x\\ \cdot&x&y\end{bmatrix}$&$[0,0,0,q+1]$\\ \hline

 \end{tabular}
 \caption{\label{tableoflinesodd}\scriptsize The $K$-orbits of lines in $\PG(5,q)$, $q$ odd. The parameters $u,v_0,v_1,v_2, \alpha, \beta, \gamma, \delta$ in $\Fq$ are defined as follows: $v_i\lambda^2+uv_i\lambda-1 \neq 0$ for all $\lambda\in \Fq$ and $i \in\{0,1,2\}$,
$\lambda^3 + \gamma \lambda^2 - \beta \lambda + \alpha \neq 0$ for all $\lambda \in \Fq$,  $-v_1 \in \square_q$, $-v_2 \not\in \square_q$ and $\delta \not\in \square_q$.}
\end{center}
\end{table}

\begin{table}[!htbp]
\begin{center}
\scriptsize

\begin{tabular}[!htbp]{llll} 
 \hline
$L^K$ & Webs of Conics &$OD_4(L)$\\ \hline
\vspace{0.2cm}
$ o_5$&  $(X_0X_1,X_0X_2,X_1X_2,X_2^2)$&$[1,2q^2+q,0,q^3-q^2]$
 \\ \vspace{0.2cm}

$ o_6$& $(X_0X_2,X_1^2,X_1X_2,X_2^2)$	 & $[q+1, \frac{3q^2+q}{2},\frac{q^2-q}{2}, q^3-q^2]$

\\ \vspace{0.2cm}

$o_{8,1}$  & $(X_0X_1,X_0X_2,X_1X_2,X_1^2+X_2^2)$ & $[2, q^2+\frac{3q-1}{2},\frac{q-1}{2}, q^3-q]$
 \\\vspace{0.2cm}

$o_{8,2}$  & $(X_0X_1,X_0X_2,X_1X_2,\delta X_1^2+X_2^2)$ & $[0,q^2+\frac{3q+1}{2},\frac{q+1}{2},q^3-q]$
\\ \vspace{0.2cm}

$o_9 $  &$(X_0X_1,X_0X_2-X_1^2,X_1X_2,X_2^2)$ &$[1,q^2+q,0,q^3]$
\\\vspace{0.2cm}

$o_{10}$ & $(v_0^{-1}X_0^2+uX_0X_1-X_1^2,X_0X_2,X_1X_2,X_2^2)$& 
$[1,q^2+q,q^2,q^3-q^2]$
 \\\vspace{0.2cm}

  $ o_{12,1}$&	$(X_0^2,X_0X_2,X_1^2,X_2^2)$ & $[q+2,q^2+\frac{q-1}{2}, q^2-\frac{q+1}{2}, q^3-q^2]$
   \\ \vspace{0.2cm}

   $ o_{13,1}$&	$(X_0^2,X_0X_2,X_1^2+X_2^2,X_1X_2)$& $[3,\frac{q^2+3q-2}{2},\frac{q^2+q-2}{2},q^3-q]$
   \\ \vspace{0.2cm}

 $o_{13,2}$ & $(X_0^2,X_0X_2,\delta X_1^2+X_2^2,X_1X_2)$& $[1,\frac{q^2+3q}{2},\frac{q^2+q}{2},q^3-q]$\\\vspace{0.2cm}

 $o_{14,1}$ & $(X_0X_1,X_0X_2,X_0^2+X_1^2+X_2^2,X_1X_2)$&$[4,\frac{q^2-1}{2}+2q-1,\frac{q^2-1}{2}+q-1,q^3-2q]$\\
  \vspace{0.2cm}

 $o_{14,2}$ & $(X_0X_1,X_0X_2,\delta X_0^2+X_1^2+\delta X_2^2,X_1X_2)$&$[0,\frac{q^2+1}{2}+2q,\frac{q^2+1}{2}+q,q^3-2q]$
\\
  \vspace{0.2cm}
    $ o_{15,1}$& $(X_0X_2,X_1X_2,X_0X_1-X_2^2,v_1^{-1}X_0^2+uX_0X_1-X_1^2)$& $[2,\frac{q^2-1}{2}+q,\frac{q^2-1}{2},q^3]$
     \\ \vspace{0.2cm}
     
     $ o_{15,2}$& $(X_0X_2,X_1X_2,X_0X_1-X_2^2,v_2^{-1}X_0^2+uX_0X_1-X_1^2)$&  $[0,\frac{q^2+1}{2}+q,\frac{q^2+1}{2},q^3]$
     \\ \vspace{0.2cm}

     $o_{16,1}$ &$(X_0^2,X_0X_1,X_0X_2-X_1^2,X_2^2)$ &$[2,\frac{q^2-1}{2}+q,\frac{q^2-1}{2},q^3]$
 \\\vspace{0.2cm}
        
    $o_{17}$&$(X_0X_2,X_0X_1-X_2^2,\alpha X_0^2-X_1X_2,\beta X_0X_1-X_1^2-\gamma X_1X_2)$ & $[1,\frac{q^2+q}{2},\frac{q^2-q}{2},q^3+q]$
    \\  \hline

   \end{tabular}

 \caption{\label{tablewebsodd}\scriptsize The $K$-orbits of lines in $\PG(5,q)$ and their equivalent webs of conics in $\PG(2,q)$, $q$ odd. The parameters $u,v_0,v_1,v_2, \alpha, \beta, \gamma, \delta$ in $\Fq$ are defined as follows: $v_i\lambda^2+uv_i\lambda-1 \neq 0$ for all $\lambda\in \Fq$ and $i \in\{0,1,2\}$,
$\lambda^3 + \gamma \lambda^2 - \beta \lambda + \alpha \neq 0$ for all $\lambda \in \Fq$,  $-v_1 \in \square_q$, $-v_2 \not\in \square_q$ and $\delta \not\in \square_q$.}

\end{center}
\end{table}

\begin{table}[!htbp]
\begin{center}
\scriptsize
\begin{tabular}[h]{l l ll} 
 \hline
$L^K$& Representatives&$OD_0(L)$\\ \hline
$o_5$&$\begin{bmatrix}x&\cdot&\cdot\\\cdot&y&\cdot\\ \cdot&\cdot&\cdot\end{bmatrix}$ &$[2,0,q-1,0]$\\
$o_6$&$\begin{bmatrix}x&y&\cdot\\y&\cdot&\cdot\\ \cdot&\cdot&\cdot\end{bmatrix}$  &$[1,1,q-1,0]$\\
$o_{8,1}$ & $\begin{bmatrix}x&\cdot&\cdot\\\cdot&y&\cdot\\ \cdot&\cdot&-y\end{bmatrix}$&$[1,0,1,q-1]$\\
$o_{8,3}$ &$\begin{bmatrix}x&\cdot&\cdot\\\cdot&\cdot&y\\ \cdot&y&\cdot\end{bmatrix}$ &$[1,1,0,q-1]$\\
$o_9$ & $\begin{bmatrix}x&\cdot&y\\\cdot&y&\cdot\\ y&\cdot&\cdot\end{bmatrix}$ &$[1,0,0,q]$\\
$o_{10}$ & $\begin{bmatrix}v_0x&y&\cdot\\y&x+uy&\cdot\\ \cdot&\cdot&\cdot\end{bmatrix}$ &$[0,0,q+1,0]$\\
$o_{12,1}$ &$\begin{bmatrix}\cdot&x&\cdot\\x&\cdot&y\\ \cdot&y&\cdot\end{bmatrix}$&$[0,q+1,0,0]$\\ 
$o_{12,3}$ & $\begin{bmatrix}\cdot&x&\cdot\\x&x+y&y\\ \cdot&y&\cdot\end{bmatrix}$&$[0,1,q,0]$\\ 
$o_{13,1}$ &$\begin{bmatrix}\cdot&x&\cdot\\x&y&\cdot\\ \cdot&\cdot&-y\end{bmatrix}$ &$[0,1,1,q-1]$\\ 
$o_{13,3}$ &$\begin{bmatrix}\cdot&x&\cdot\\x&x+y&\cdot\\ \cdot&\cdot&y\end{bmatrix}$ &$[0,0,2,q-1]$\\ 
$o_{14,1}$ & $\begin{bmatrix}x&\cdot&\cdot\\\cdot&-(x+y)&\cdot\\ \cdot&\cdot&y\end{bmatrix}$&$[0,0,3,q-2]$\\
$o_{15,1}$ &$\begin{bmatrix}v_1y&x&\cdot\\x&ux+y&\cdot\\ \cdot&\cdot&x\end{bmatrix}$ &$[0,0,1,q]$ \\ 
$o_{16,1}$ & $\begin{bmatrix}\cdot&\cdot&x\\\cdot&x&y\\ x&y&\cdot\end{bmatrix}$&$[0,1,0,q]$\\
$o_{16,3}$ &$\begin{bmatrix}\cdot&\cdot&x\\\cdot&x&y\\ x&y&y\end{bmatrix}$& $[0,0,1,q]$\\  
$o_{17}$ & $\begin{bmatrix}\alpha^{-1}x&y&\cdot\\y&\beta y-\gamma x&x\\ \cdot&x&y\end{bmatrix}$&$[0,0,0,q+1]$\\ \hline

 \end{tabular}
 \caption{\label{tableoflineseven} \scriptsize The $K$-orbits of lines in $\PG(5,q)$, $q$ even. The parameters $u,v_0,v_1,\alpha, \beta, \gamma$ in $\Fq$ are defined as follows: $v_i\lambda^2+uv_i\lambda-1 \neq 0$ for all $\lambda\in \Fq$ and $i \in\{0,1\}$ and
$\lambda^3 + \gamma \lambda^2 - \beta \lambda + \alpha \neq 0$ for all $\lambda \in \Fq$.}
\end{center}
\end{table}

\begin{table}[!htbp]
\begin{center}
\scriptsize

\begin{tabular}[!htbp]{ l ll l} 
 \hline
$L^K$ & Webs of Conics &$OD_4(L)$\\ \hline
\vspace{0.2cm}

$ o_5$&  $(X_0X_1,X_0X_2,X_1X_2,X_2^2)$&$[1,2q^2+q,0,q^3-q^2]$
 \\ \vspace{0.2cm}
 
$ o_6$& $(X_0X_2,X_1^2,X_1X_2,X_2^2)$	 & $[q+1, \frac{3q^2+q}{2},\frac{q^2-q}{2}, q^3-q^2]$

\\ \vspace{0.2cm}

$o_{8,1}$  & $(X_0X_1,X_0X_2,X_1X_2,X_1^2+X_2^2)$ &$[1,q^2+\frac{3}{2}q,\frac{q}{2},q^3-q]$
 \\\vspace{0.2cm}

$o_{8,3}$  & $(X_0X_1,X_0X_2,X_1^2,X_2^2)$ &$[q+1,q^2+q,0,q^3-q]$
\\ \vspace{0.2cm}

$o_9 $  &$(X_0X_1,X_0X_2+X_1^2,X_1X_2,X_2^2)$ &$[1,q^2+q,0,q^3]$
\\\vspace{0.2cm}

$o_{10}$ & $(v_0^{-1}X_0^2+uX_0X_1+X_1^2,X_0X_2,X_1X_2,X_2^2)$& $[1,q^2+q,q^2,q^3-q^2]$
 \\\vspace{0.2cm}

  $ o_{12,1}$&	$(X_0^2,X_0X_2,X_1^2,X_2^2)$ &$[q^2+q+1,\frac{q^2+q}{2},\frac{q^2-q}{2},q^3-q^2]$
   \\ \vspace{0.2cm}

   $ o_{12,3}$&	$(X_0^2,X_0X_2,X_0X_1+X_1X_2+X_1^2,X_2^2)$& $[q+1,q^2+\frac{q}{2}, q^2-\frac{q}{2}, q^3-q^2]$
   \\ \vspace{0.2cm}

   $ o_{13,1}$&	$(X_0^2,X_0X_2,X_1^2+X_2^2,X_1X_2)$& $[q+1,\frac{q^2}{2}+q,\frac{q^2}{2},q^3-q]$ 
   \\ \vspace{0.2cm}

 $o_{13,3}$ & $(X_0^2,X_0X_2,X_1^2+X_0X_1+X_2^2,X_1X_2)$& $[1,\frac{q^2+3q}{2},\frac{q^2+q}{2},q^3-q]$
 \\\vspace{0.2cm}

 $o_{14,1}$ & $(X_0X_1,X_0X_2,X_0^2+X_1^2+X_2^2,X_1X_2)$&$[1,\frac{q^2}{2}+2q,\frac{q^2}{2}+q,q^3-2q]$\\
  \vspace{0.2cm}

    $ o_{15,1}$& $(X_0X_2,X_1X_2,X_0X_1+X_2^2,v_1^{-1}X_0^2+uX_0X_1+X_1^2)$& $[1,\frac{q^2}{2}+q,\frac{q^2}{2},q^3]$ 
     \\ \vspace{0.2cm}

     $o_{16,1}$ &$(X_0^2,X_0X_1,X_0X_2+X_1^2,X_2^2)$ & $[q+1,\frac{q^2+q}{2}, \frac{q^2-q}{2},q^3]$
 \\\vspace{0.2cm}

     $o_{16,3}$ & $(X_0^2,X_0X_1,X_0X_2+X_1^2,X_1X_2+X_2^2)$&$[1,\frac{q^2}{2}+q,\frac{q^2}{2},q^3]$
 \\\vspace{0.2cm}

    $o_{17}$&$(X_0X_2,X_0X_1+X_2^2,\alpha X_0^2+X_1X_2,\beta X_0X_1+X_1^2+\gamma X_1X_2)$ & $[1,\frac{q^2+q}{2},\frac{q^2-q}{2},q^3+q]$
    \\  \hline

   \end{tabular}

 \caption{\label{tablewebseven}\scriptsize The $K$-orbits of lines in $\PG(5,q)$ and their equivalent webs of conics in $\PG(2,q)$, $q$ even. The parameters $u,v_0,v_1,\alpha, \beta, \gamma$ in $\Fq$ are defined as follows: $v_i\lambda^2+uv_i\lambda-1 \neq 0$ for all $\lambda\in \Fq$ and $i \in\{0,1\}$ and
$\lambda^3 + \gamma \lambda^2 - \beta \lambda + \alpha \neq 0$ for all $\lambda \in \Fq$.}

\end{center}
\end{table}

\begin{table}[!htbp]
\begin{center}
\scriptsize
\begin{tabular}[h]{ l l l} 
 \hline

 $S^K$ &Representatives   &Conditions\\ \hline\\
   \vspace{3pt}
 $ \Omega_5$  &		$\begin{bmatrix} .&x&y\\x&.&z\\y&z&t           \end{bmatrix}$   &\\
  \vspace{3pt}

$\Omega_6$  &$\begin{bmatrix} .&.&x\\.&y&z\\x&z&t           \end{bmatrix}$   & \\ 

\vspace{3pt}

$\Omega_{8,1}$  &$\begin{bmatrix} .&x&y\\x&z&t\\y&t&z        \end{bmatrix}$  & \\

\vspace{3pt}

 $\Omega_{8,2}$  &$\begin{bmatrix} .&x&y\\x&\gamma z&t\\y&t&z           \end{bmatrix}$  & $\gamma \not\in \square_q$\\

\vspace{3pt}

  $\Omega_{9}$  &$\begin{bmatrix} .&x&y\\x&-y&z\\y&z&t           \end{bmatrix}$ &\\

\vspace{3pt}

  $\Omega_{10}$  &$\begin{bmatrix} x&uvx&y\\uvx&-vx&z\\y&z&t          \end{bmatrix}$   &$(*)$\\

 \vspace{3pt}

$\Omega_{12}$ &  $\begin{bmatrix} x&.&y\\.&z&.\\y&.&t           \end{bmatrix}$  
&\\

\vspace{3pt}

$\Omega_{13,1}$ &$\begin{bmatrix} x&.&y\\.&z&t\\y&t&z          \end{bmatrix}$   &\\

\vspace{3pt}

$\Omega_{13,2}$ &  $\begin{bmatrix} x&.&y\\ .&\gamma z&t\\y&t&z \end{bmatrix}$ 
 & $\gamma \not\in \square_q$\\

\vspace{3pt}

$\Omega_{14,1}$ & $\begin{bmatrix} x&y&z\\y&x&t\\z&t&x           \end{bmatrix}$  &\\ 

\vspace{3pt}

$\Omega_{14,2}$ &  
$\begin{bmatrix} \gamma x&y&z\\y&x&t\\z&t&\gamma x           \end{bmatrix}$ &  $\gamma \not\in \square_q$\\  

\vspace{3pt}

$\Omega_{15,1}$ &  $\begin{bmatrix} x&y&z\\y&-v_1x&t\\z&t&-y+uv_1x          \end{bmatrix}$ 
 &$(*)$, $-v_1 \in \square_q$\\

\vspace{3pt}

$\Omega_{15,2}$ &  $\begin{bmatrix} x&y&z\\y&-v_2x&t\\z&t&-y+uv_2x          \end{bmatrix}$ 
 &$(*)$, $-v_2 \not\in \square_q$\\

 \vspace{3pt}
$\Omega_{16}$ & $\begin{bmatrix} x&y&z\\y&-z&.\\z&.&t           \end{bmatrix}$  
  & \\ 

\vspace{3pt}

$\Omega_{17}$ &  $\begin{bmatrix} \alpha\gamma z-\alpha t&x&y\\x&z&t\\y&t&-x-\beta z \end{bmatrix}$    &$(**)$\\

\hline
 \end{tabular}
 \caption{\label{solidsodd}{\scriptsize The $K$-orbits of solids in $\PG(5,q)$, $q$ odd, and their representatives. Condition $(*)$ is: $v\lambda^2+uv\lambda-1 \neq 0$ for all $\lambda\in \Fq$, where $v\in \{v_1, v_2\}$ in the case $\Omega_{15}$. Condition $(**)$ 
is: $\lambda^3 + \gamma \lambda^2 - \beta \lambda + \alpha \neq 0$ for all $\lambda \in \Fq$.}}

\end{center}

\end{table}

\end{appendices}


\begin{thebibliography}{10}

\bibitem{AbEmIa} N. Abdallah, J. Emsalem and A. Iarrobino, ``Nets of Conics and associated Artinianalgebras of length 7", European Journal of Mathematics. \textbf{9}, (2023), 238583607.

\bibitem{T233paper}
N. Alnajjarine, M. Lavrauw, “ Determining the rank of tensors in $\mathbb{F}_2^q\otimes \mathbb{F}_3^q\otimes \mathbb{F}_3^q$" , in: D. Slamanig, E. Tsigaridas, Z. Zafeirakopoulos (Eds.), MACIS 2019: Mathematical Aspects of Computer and Information Sciences, in: Lecture Notes in Computer Science, vol.11989, Springer, Cham, 2020.

\bibitem{solidsqeven}
N. Alnajjarine, M. Lavrauw, T. Popiel, “ Solids in the space of the Veronese surface in even characteris-
tic.”Finite Fields and Their Applications. \textbf{83}, (2022), 102068.

\bibitem{planesqeven}
N. Alnajjarine, M. Lavrauw, “ A classification of planes intersecting the Veronese surface
over finite fields of even order.”Designs, Codes and Cryptography. \textbf{83}, (2023), https://doi.org/10.1007/s10623-023-01194-9.






\bibitem{fining} J.~Bamberg, A.~Betten, Ph.~Cara, J.~De~Buele, M.~Lavrauw and M.~Neunh\"offer, FinInG: Finite Incidence Geometry: FinInG -- A {GAP} package, (2018); \url{http://www.fining.org}



\bibitem{roots} E.~Berlekamp, H.~Rumsey and G.~Solomon G, ``On the solution of algebraic equations over finite fields'', {\it Information and Control} \textbf{10}, (1967), 553--564.



\bibitem{campbell}
A.~Campbell, ``Pencils of conics in the Galois fields of order $2^n$'', {\it Amer. J. Math.} \textbf{49}, (1927), 401--406.

\bibitem{campbell2}
A. Campbell, “Nets of conics in the Galois field of order $2^n$”, Bull. Amer. Math. Soc. \textbf{34}, (1928), 481--489.
 
\bibitem{dickson}
L.~E.~Dickson, ``On families of quadratic forms in a general field'', {\it Quarterly J. Pure Appl. Math.} \textbf{45}, (1908), 316--333.


  
\bibitem{GAP} The GAP Group, GAP -- Groups, Algorithms, and Programming, Version 4.11.1, (2021); \newline \url{https://www.gap-system.org}.
  
  


\bibitem{galois geometry} J. Hirschfeld and J. Thas, \newblock {\it  General Galois geometries.} \newblock London : Springer, (1991).


\bibitem{hirsch}
J.~W.~P.~Hirschfeld, {\it Projective geometries over finite fields}, second edition, Oxford University Press, Oxford, (1998).




\bibitem{jordan1}
C.~Jordan, ``R\'eduction d’un r\'eseau de formes quadratiques ou bilin\'eaires: premi\`ere partie'', {\it J. Math. Pures Appl.} (1906), 403--438.

\bibitem{jordan2}
C.~Jordan, ``R\'eduction d’un r\'eseau de formes quadratiques ou bilin\'eaires: deuxi\`eme partie'', {\it J. Math. Pures Appl.} (1907), 5--51. 


\bibitem{Lavrauw2011} M.~Lavrauw,  ``Finite semifields and non-singular tensors", {\it Des. Codes Cryptogr.} \textbf{68}, (2013), 205--227.

\bibitem{lines} 
M.~Lavrauw and T.~Popiel, ``The symmetric representation of lines in $\mathrm{PG}(\Fq^3\otimes\Fq^3)$'', {\it Discrete Math.} \textbf{343}, (2020), 111775.

\bibitem{nets}
M.~Lavrauw, T.~Popiel and J.~Sheekey, ``Nets of conics of rank one in $\mathrm{PG}(2,q)$, $q$ odd'', {\it J. Geom.} \textbf{111}, (2020), 36. 
 
\bibitem{LaPoSh2021}
M.~Lavrauw, T.~Popiel and J.~Sheekey, ``Combinatorial invariants for nets of conics in $\PG(2,q)$'', {\it Des. Codes. Cryptogr.} (2021), https://doi.org/10.1007/s10623-021-00881-9.

\bibitem{canonical}
M.~Lavrauw and J.~Sheekey, ``Canonical forms of $2 \times 3 \times 3$ tensors over the real field, algebraically closed fields, and finite fields'', {\it Linear Algebra Appl.} \textbf{476}, (2015), 133--147.


\bibitem{segre}
C. Segre, "Studio sulle quadriche
in uno spazio lineare ad un numero qualunque di dimensioni'', Mem. R. Acc. Scienze Torino \textbf{36}
(1883), 3--86 (italian).


\bibitem{wall}
C. Wall, ``Nets of conics'', {\it Math. Proc. Cambridge Philos. Soc.} \textbf{81}, (1977), 351--364.


\bibitem{wilson}
A. Wilson, “The canonical Types of Nets of Modular Conics”, Amer. J. Math. \textbf{36}, (1914), 187--210.




\end{thebibliography}
\end{document}